\documentclass[12pt]{amsart}
\usepackage{amsmath,amsfonts,latexsym,graphicx,amssymb,url}
\usepackage{amsmath,txfonts,pifont,bbding,pxfonts,manfnt}
\usepackage{psfrag}
\usepackage{hyperref}
\usepackage[active]{srcltx}
\usepackage{pdfsync}
\newcommand{\dist}{\text{dist}} %%distance

\newcommand{\R}{{\mathbb R}} %%reals

\theoremstyle{plain}
\newtheorem{theorem}{Theorem}[section]
\newtheorem{corollary}[theorem]{Corollary}
\newtheorem{lemma}[theorem]{Lemma}
\newtheorem{proposition}[theorem]{Proposition}
\newtheorem{definition}[theorem]{Definition}
\newtheorem{remark}[theorem]{Remark}

\providecommand{\bysame}{\makebox[3em]{\hrulefill}\thinspace}
\begin{document}

\setcounter{equation}{0}

%\documentstyle[12pt]{article}

%\usepackage{amsmath}
%\usepackage{amsfonts}
%\usepackage{amssymb}
%\usepackage{amscd}

%\setlength{\headheight}{15pt}
%\setlength{\topmargin}{10pt}
%\setlength{\headsep}{30pt}
%\setlength{\textwidth}{15cm}
%\setlength{\textheight}{21.5cm}
%\setlength{\oddsidemargin}{1cm}
%\setlength{evensidemargin}{1cm}

%\newtheorem{theorem}{Theorem}[section]

%\newtheorem{lemma}[theorem]{Lemma}

%\newtheorem{corollary}[theorem]{Corollary}

%\newtheorem{proposition}[theorem]{Proposition}

%\newtheorem{definition}[theorem]{Definition}

%\begin{document}

\title[Regularity of refractors]{Regularity for the near field parallel refractor and reflector problems} 
	
\author[C. E. Guti\'errez and F. Tournier]{Cristian E. Guti\'errez and Federico Tournier}
\thanks{\today.\\ The first author was partially supported
by NSF grant DMS--1201401.}
\address{Department of Mathematics\\Temple University\\Philadelphia, PA 19122}
\email{gutierre@temple.edu}
\address{Instituto Argentino de Matem\'atica A. P. Calder\'on\\CONICET\\Saavedra 15, Buenos Aires (CP 1083), Argentina}
\email{fedeleti@aim.com}

\begin{abstract}
We prove local $C^{1,\alpha}$ estimates of solutions for the parallel refractor and reflector problems under local assumptions on the target set $\Sigma$, and no assumptions are made on the smoothness of the densities.
\end{abstract}
\maketitle
%\dominitoc
{\tiny\tableofcontents}

\setcounter{equation}{0}
\section{Introduction}
%The problem considered in this paper is the following.
Suppose we have a domain $\Omega\subset \R^n$ and 
a domain $\Sigma$ contained in an $n$ dimensional surface in $\R^{n+1}$; 
$\Sigma$ is referred as the target domain or screen to be illuminated.
%(for the practical applications,
%one can think that $n=3$).
Let $n_1$ and $n_2$ be the indexes of refraction
of two homogeneous and isotropic media I and II, respectively,
and suppose that from the region $\Omega$ surrounded by medium I,
radiation emanates in the vertical direction $e_{n+1}$ with intensity $f(x)$
for $x\in \Omega$, and $\Sigma$ is surrounded by media II.
That is, all emanating rays from $\Omega$ are collimated.
A {\it parallel refractor} is an optical surface
$\mathcal R$ 
interface between media I and II,
such that all rays refracted by $\mathcal R$ into medium II
are received at the surface $\Sigma$, and the prescribed radiation intensity
received at each point $p\in \Sigma$ is $\sigma(p)$.
%Of course some conditions on the relative position of $\Sigma$ and and $\Omega$ are needed so rays can be refracted to $\Sigma$,
%see conditions (A) and (B) below.
Assuming no loss of energy in this process, we have the conservation of energy equation
$\int_{\Omega}f(x)\,dx=\int_{\Sigma}\sigma(p)\,dp$.
Under general conditions on $\Omega$ and $\Sigma$, and when $\sigma$ is a Radon measure in $D$, the existence of parallel refractors is proved in \cite{gutierrez-tournier:parallelrefractor}. 

The purpose of this paper 
is to study local regularity of parallel refractors and reflectors. Indeed, under suitable conditions on the target and the measure $\sigma$, we prove local $C^{1,\alpha}$ estimates.
More precisely, if $u$ is a parallel refractor in $\Omega$, the target $\Sigma$ satisfies the local condition \eqref{eq:AWPL} from $(x^\star,u(x^\star))$, 
%the local condition \eqref{eq:convexityconditiononthetargetbisbis} holds, 
%the connectedness condition \eqref{eq:curveisaconnectedset} holds, 
and the measure $\sigma$ satisfies a local condition
at that point, condition \eqref{eq:conditiononsigmameasureofthetube}, then $u\in C^{1,\alpha}$ in a neighborhood of the point $x^\star$. 

Throughout the paper we assume that media $II$ is denser than media $I$, that is, $n_{1}<n_{2}$.
When $n_{1}>n_{2}$, the geometry of the refractor changes. One needs to use hyperboloids of revolution instead of ellipsoids as 
indicated in \cite{gutierrez-huang:farfieldrefractor}.

The plan of the paper is as follows. Section \ref{sec:preliminaries} contains results concerning estimates of ellipsoids of revolution, and Subsection \ref{subsec:assumptiontargetrefractor} contains basic assumptions on the target. Section \ref{sec:regularityhypothesesonthetarget} contains assumptions on the target modeled on the conditions introduced by 
Loeper \cite[Proposition 5.1]{loeper:actapaper}. Indeed, we assume the target satisfies the local condition 
\eqref{eq:convexityconditiononthetargetbisbis}. We also introduce the differential condition \eqref{eq:AWP}, similar in form to condition (A3) of Ma, Trudinger and Wang \cite{MaTrudingerWang:regularityofpotentials}, and show in 
Theorem \ref{thm:awforperpendicularvectors} and Remark \ref{rmk:localequivalenceperpvectors} that \eqref{eq:AWPL} and \eqref{eq:convexityconditiononthetargetbisbis} are equivalent.
In Section \ref{subsec:localandglobalrefractors}, we prove that if an ellipsoid supports a  parallel refractor 
locally, then it supports the refractor globally provided the target satisfies the condition (AW) given in \eqref{eq:AWPLnew}. The main result in this section is Proposition \ref{prop:localimpliesglobal} used later in the proof of Theorem \ref{thm:mainestimate}. 
Section \ref{sec:definitionofrefractorandmainresults} contains the main results, Lemma \ref{lm:mainlemmarefractor} and Theorem \ref{thm:mainestimate},
and also Proposition \ref{prop:mu(S)=0} used later for the application of these results to show regularity of parallel refractors constructed in  
\cite{gutierrez-tournier:parallelrefractor}, Corollary \ref{corollary:applicationtorefractors}.
Section \ref{sec:holderregularityofgradient} contains H\"older estimates of gradients of refractors under the assumptions  
\eqref{eq:conditiononsigmameasureofthetube} and \eqref{eq:assumptiononmeasuresigma}  on the target $\Sigma$ and the measure $\sigma$ on $\Sigma$.
Section \ref{sec:example} contains examples of targets verifying the assumptions, see condition \eqref{eq:quantitativeconditiononthetarget}.
Up to this point in the paper, refractors are defined with ellipsoids supporting the refractor from above. Refractors can also be defined with ellipsoids supporting from 
below, and in Section \ref{sec:alternativedefinitionofrefractor} we obtain the same regularity results for refractors with this definition. 
In Section \ref{sec:regularityforthereflectorproblem} we obtain similar regularity results for the {\it near field parallel reflector} problem. In this case the proofs are simpler because the differential condition \eqref{eq:AWPLreflectors}
implies the global inequality \eqref{eq:reflectorconvexityconditiononthetargetbisbis}.
%Finally in Section \ref{sec:appendix}, we discuss the issue of local and global supporting ellipsoids.

{\bf Acknowledgements}. It is a pleasure to thank Neil Trudinger and Philippe Delano\"e for useful comments and suggestions.

\setcounter{equation}{0}
\section{Preliminaries}\label{sec:preliminaries}

\subsection{Refraction}
We briefly review the process of refraction.
%Fix $0<k<1$. 
Our setting is $\R^{n+1}$. 
Points will be denoted by $X=(x,x_{n+1})$. We consider parallel rays traveling in the unit direction $e_{n+1}$. Let $T$  be a hyperplane with upper unit normal $N$ and $X\in T$. We assume that the region below $T$ has refractive index $n_1$ and the region above $T$ has refractive index $n_2$ and $\kappa:=\dfrac{n_1}{n_2}<1$, e.g., air to glass. In such case, by Snell's law of refraction, a ray coming from below in direction $e_{n+1}$ that hits $T$ at $X$ is refracted in the unit direction 
$$
\Lambda=\kappa\, e_{n+1} +\delta N, \quad \text{with}\quad 
\delta=
%-\kappa\,e_{n+1}\cdot N  +\sqrt{\kappa^{2}\,(e_{n+1}\cdot N)^{2}+1-\kappa^2},
-\kappa\,e_{n+1}\cdot N  +\sqrt{1+\kappa^{2}\,\left((e_{n+1}\cdot N)^{2}-1\right)},$$ 
where $\delta>0$ since $\kappa<1$.
%HERE IT SHOULD BE
%\[
%\Lambda=\dfrac{1}{\kappa}(e_{n+1}-\delta\, N),\qquad
%\delta=e_{n+1}\cdot N-\kappa \sqrt{1-\kappa^{-2}(1-(e_{n+1}\cdot N)^2)}.
%\]

\begin{figure}[htp]
\begin{center}
   \includegraphics[width=2in]{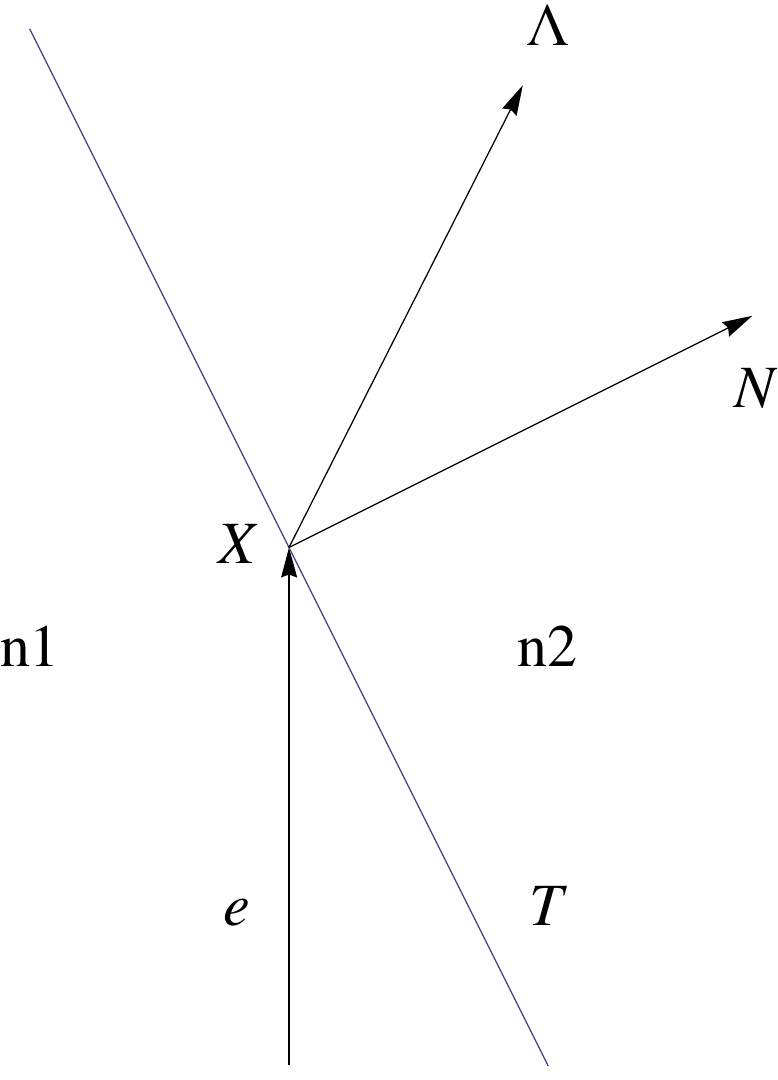}
    %\subfigure[$|X|+2/3 |X-P|=1.7$, $P=(2,0)$]{\label{fig:edge-b}\includegraphics[width=2.9in]{snelllaw.pdf}} 
%    \subfigure[After Sobel edge detection]{\label{fig:edge-c}\includegraphics[scale=1]{SeveralCartesianOvalskappa<1.pdf}}
\end{center}
  \caption{Snell's law, $n_1<n_2$}
  \label{fig:snelllaw}
\end{figure}

%\begin{figure}[htp]
%\begin{center}
%    \subfigure[$|X|+2/3 |X-P|=1.4-1.9$, $P=(2,0)$]{\label{fig:edge-a}\includegraphics[width=2in]{snelllaw.pdf}}
%    %\subfigure[$|X|+2/3 |X-P|=1.7$, $P=(2,0)$]{\label{fig:edge-b}\includegraphics[width=2.9in]{snelllaw.pdf}} 
%%    \subfigure[After Sobel edge detection]{\label{fig:edge-c}\includegraphics[scale=1]{SeveralCartesianOvalskappa<1.pdf}}
%\end{center}
%  %\caption{Cartesian ovals $\kappa<1$, e.g., glass to air}
%  \label{snelllaw}
%\end{figure}
%\centerline{\includegraphics[angle=0,width=52mm]{snelllaw.pdf}}

In particular, if $v\in \R^{n}$ and the hyperplane $T$ is so that the unit upper normal $N=\dfrac{(-v,1)}{\sqrt{1+|v|^2}}$, then the refracted unit direction is 
\begin{equation}\label{eq:defofLambdav}
\Lambda (v):=\left(-\dfrac{\delta}{\sqrt{1+|v|^2}}\, v, \dfrac{\delta}{\sqrt{1+|v|^2}}+\kappa \right):=(-Q(v)v,Q(v)+\kappa),
\end{equation}
with $Q(v)^{2}|v|^{2}+(Q(v)+\kappa)^{2}=1$ and $Q>0$\footnote{We have in this case, 
$\delta=\dfrac{-\kappa+\sqrt{1+(1-\kappa^2)|v|^2}}{\sqrt{1+|v|^2}}$.}. 
The refracted ray is $X+s\Lambda$, for $s>0$.
Here we have in mind that $T$ is the tangent plane to a refractor $u$ at $(x,u(x))$ and so $v=Du(x)$.

\subsection{Ellipsoids}

Given $b>0$ and $Y=(y,y_{n+1})$, consider  
$$E(Y,b)=\{X:|X-Y|+k(x_{n+1}-y_{n+1})=b\}.$$
$E(Y,b)$ is an ellipsoid of revolution with foci $(y,y_{n+1})$,  $\left(y,y_{n+1}-\dfrac{2\,\kappa \,b}{1-\kappa^2}\right)$.
The projection of $E(Y,b)$ over $\R^n$ is the ball $B_{b/\sqrt{1-\kappa^2}}(y)$.
We define the lower and upper parts of $E(Y,b)$ by
\begin{align*}
E^-(Y,b)&=\left\{ X\in E(Y,b): x_{n+1}\leq y_{n+1}-\dfrac{\kappa \, b}{1-\kappa^2}\right\},\\
E^+(Y,b)&=\left\{ X\in E(Y,b): x_{n+1}\geq y_{n+1}-\dfrac{\kappa \, b}{1-\kappa^2}\right\},
\end{align*}
respectively. We can regard $E^-(Y,b)$ as the graph of the function 
\begin{equation}\label{eq:functiondefininglowerellipsoid}
\phi(x)=y_{n+1}-\dfrac{\kappa\,b}{1-k^{2}}-\sqrt{\dfrac{b^{2}}{(1-\kappa^{2})^{2}}-\dfrac{|x-y|^{2}}{1-k^{2}}},
\end{equation} 
for $x\in B_{b/\sqrt{1-\kappa^2}}(y)$.
If $X=(x,\phi(x))$, $x_{n+1}=\phi(x)$, then $D_x\phi(x)=\dfrac{x-y}{y_{n+1}-x_{n+1}-k|X-Y|}$. 
A ray with direction $e_{n+1}$ that hits from below the graph of $\phi$ at the point $X=(x,\phi(x))$ is refracted along  the ray $X+s\Lambda(v)$ with $v=D\phi(x)$, and therefore it passes through the focus $Y$. 

If $X=(x,\phi(x))$ and the focus $Y$ can be written as $Y=X+s\,\Lambda(v)$ for some $s>0$ and $v\in \R^{n}$, then $v=D\phi(x)$. 
%I BELIEVE THIS IS MORE CLEAR THAN THE SENTENCE BELOW, IS THIS WHAT YOU WANT TO SAY?
%
%Let $v\in R^{n}$, fix $Y=X+s\Lambda(v)$ where $X=(x,\phi(x))$ . Then $D\phi(x)=v$, and this means that the ray that hits the graph of $\phi$ at the point $X$ gets refracted along  the ray that passes through the point $Y$. IS THIS CORRECT? IS $v$  ARBITRARY?

Given $X,Y\in \R^{n+1}$, let us define 
\begin{equation}\label{eq:definitionofc(X,Y)}
c(X,Y)=|X-Y|+k(x_{n+1}-y_{n+1}),
\end{equation}
and we have 
\begin{equation}\label{eq:estimatesforc(X,Y)}
(1-\kappa)\,|X-Y|\leq c(X,Y)\leq (1+\kappa)\,|X-Y|.
\end{equation}
Given $X_0,Y\in \R^{n+1}$ with $X_0\neq Y$, then $E(Y,c(X_0,Y))$ is the unique ellipsoid passing through $X_0$ and having upper focus at $Y$.

We have that 
\begin{equation}\label{eq:lowerellipsoid}
\text{$X\in E^-(Y,b)$ iff $y_{n+1}-x_{n+1}-\kappa\,|X-Y|\geq 0$}.
\end{equation}
and 
\[
\text{$X\in E^+(Y,b)$ iff $y_{n+1}-x_{n+1}-\kappa\,|X-Y|\leq 0$}.
\]

\subsection{Set up}

We fix a bounded domain $\Omega\subset \R^{n}$ and a cylinder set $C_{\Omega}=\Omega\times (0,M)$; points in $C_{\Omega}$ are denoted by the letter $X$.
% and a target set $\Sigma\subseteq \R^{n+1}$. 
%Points in $\Sigma$ will be denoted by the letter $Y$ and points in $C_{\Omega}$ by the letter $X$.
%We assume that $\dist(C_\Omega,\Sigma)>0$. 
%I THINK WE NEED TO ASSUME THAT $\Sigma$ IS CONTAINED IN THE HALF SPACE $x_{n+1}>M$.

%Let us define $c(X,Y)=|X-Y|+k(x_{n+1}-y_{n+1})$.
%Given $X_0\in C_\Omega$ and $Y\in \Sigma$, there exists a unique ellipsoid $E(Y,b)$ with upper focus $Y$ and passing through $X_0$, and so $b=c(X_0,Y)$.

Given $0<\delta<1$ we define the region
\begin{equation}\label{eq:definitionofregionT}
\mathcal T
=
\left\{ Y\in \R^{n+1}:\text{for each $X_0\in C_\Omega$ we have $X_0\in E^-(Y,c(X_0,Y))$ and $\Omega\subset 
B_{\delta\,c(X_0,Y)/\sqrt{1-\kappa^2}}(y)$}\right\}.
\end{equation}
The region $\mathcal T$ is open and unbounded. Notice that if $(y,y_{n+1})\in \mathcal T$, then $(y,y_{n+1}')\in \mathcal T$, for all $y_{n+1}'\geq y_{n+1}$. Also, if $Y\in \mathcal T$ and $X_0\in C_\Omega$, then $c(X_0,Y)\geq c_n\,|\Omega|^{1/n}\,\sqrt{1-\kappa^2}/\delta$ and so from \eqref{eq:estimatesforc(X,Y)} we get 
\begin{equation}\label{eq:distancefromTtoCOmega}
\dist(\mathcal T,C_\Omega)\geq C(\delta,\kappa,|\Omega|,n)>0.
\end{equation}

% CHECK. CAN WE GIVE MORE PROPERTIES OF THIS SET?

We have that the function
$$
\phi(x,Y,X_0)= y_{n+1}-\frac{k\,c(X_0,Y)}{1-k^{2}}-\sqrt{\frac{c(X_0,Y)^{2}}{(1-k^{2})^{2}}-\frac{|x-y|^{2}}{1-k^{2}}},
$$ 
$X_{0}=(x_{0},x_{0_{n+1}})$, defines the lower part of the ellipsoid $E(Y,c(X_0,Y))$ for $x\in B_{c(X_0,Y)/\sqrt{1-\kappa^2}}(y)$.
We then have for $X_0\in C_\Omega$, $Y\in \mathcal T$ and $X\in E^-(Y,c(X_0,Y))$, $X=(x,x_{n+1})$, with $x\in \Omega$ that
\begin{equation}\label{eq:upperestimatecwithdelta}
y_{n+1}-x_{n+1}-\kappa \,|X-Y|\geq \beta \,c(X,Y),
\end{equation}
 with $\beta=\sqrt{1-\delta^2}$.

\subsection{Estimates for the derivatives of $\phi$}\label{subsec:estimatesderivativesofphi}

We show that the function $\phi(x,Y,X_0)$ is differentiable of any order in all variables for $x\in \Omega, Y\in \mathcal T$, and $X_0\in C_\Omega$.
Let us calculate first $\dfrac{\partial \phi}{\partial x_i}$, $1\leq i \leq n$.
We have 
\[
\dfrac{\partial \phi}{\partial x_i}(x,Y,X_0)=
\dfrac{1}{1-\kappa^2}\,\dfrac{x_i-y_i}{\sqrt{\dfrac{c(X_0,Y)^{2}}{(1-k^{2})^{2}}-\dfrac{|x-y|^{2}}{1-k^{2}}} }.
\]
We have $c(X,Y)=|X-Y|+\kappa\,(x_{n+1}-y_{n+1})=c(X_0,Y)$, with $X=(x,\phi(x,Y,X_0))$, $x\in \Omega$, then
\begin{align}\label{eq:formulaforsquareroot}
y_{n+1}-x_{n+1}-\kappa\,|X-Y|
%&=
%y_{n+1}-x_{n+1}-\kappa\,c(X_0,Y) +\kappa^2\, (x_{n+1}-y_{n+1})\notag \\
&=
(1-\kappa^2)\,(y_{n+1}-x_{n+1})-\kappa\,c(X_0,Y)\notag\\
%&=
%(1-\kappa^2)\,\left(\dfrac{\kappa\,c(X_0,Y)}{1-k^{2}}+\sqrt{\dfrac{c(X_0,Y)^{2}}{(1-\kappa^{2})^{2}}-\dfrac{|x-y|^{2}}{1-k^{2}}} \right)
%-\kappa\, c(X_0,Y)\notag\\
&=
(1-\kappa^2)\,\sqrt{\dfrac{c(X_0,Y)^{2}}{(1-\kappa^{2})^{2}}-\dfrac{|x-y|^{2}}{1-k^{2}}}\geq \sqrt{1-\delta^2}\,c(X_0,Y),
\end{align}
for all $x\in \Omega$ from \eqref{eq:upperestimatecwithdelta}.
Therefore
$
\left|\dfrac{\partial \phi}{\partial x_i}(x,Y,X_0)\right|
\leq 
\dfrac{1}{\sqrt{1-\delta^2}}\,\dfrac{|x_i-y_i|}{c(X_0,Y)}.
$
Let us now calculate $\dfrac{\partial^2 \phi}{\partial x_i\partial x_j}(x,Y,X_0)$, $1\leq i,j\leq n$.
We have
\begin{equation}\label{eq:formulasecondderivativesofphi}
\dfrac{\partial^2 \phi}{\partial x_i\partial x_j}(x,Y,X_0)
=
\dfrac{\delta_{ij}\sqrt{\dfrac{c(X_0,Y)^{2}}{(1-k^{2})^{2}}-\dfrac{|x-y|^{2}}{1-k^{2}}}+ \dfrac{x_i-y_i}{1-\kappa^2}\dfrac{x_j-y_j}{\sqrt{\dfrac{c(X_0,Y)^{2}}{(1-k^{2})^{2}}-\dfrac{|x-y|^{2}}{1-k^{2}}} } }{(1-\kappa^2)\left(\sqrt{\dfrac{c(X_0,Y)^{2}}{(1-k^{2})^{2}}-\dfrac{|x-y|^{2}}{1-k^{2}}}\right)^2}.
\end{equation}
So
\begin{align*}
\left| \dfrac{\partial^2 \phi}{\partial x_i\partial x_j}(x,Y,X_0)\right|
%&\leq
%\dfrac{1}{(1-\kappa^2)\sqrt{\dfrac{c(X_0,Y)^{2}}{(1-k^{2})^{2}}-\dfrac{|x-y|^{2}}{1-k^{2}}}}
%+
%\dfrac{|x-y|^2}{(1-\kappa^2)^2 \left( \sqrt{\dfrac{c(X_0,Y)^{2}}{(1-k^{2})^{2}}-\dfrac{|x-y|^{2}}{1-k^{2}}}\right)^3}\\
&\leq 
\dfrac{1}{\sqrt{1-\delta^2}\,c(X_0,Y)}+ \dfrac{(1-\kappa^2)\,|x-y|^2}{\left(\sqrt{1-\delta^2}\,c(X_0,Y) \right)^3}
\leq
\dfrac{C(\kappa,\delta)}{c(X_0,Y)}\leq C
\end{align*}
for $x\in \Omega$, $Y\in \mathcal T$ and $X_0\in C_\Omega$ by \eqref{eq:distancefromTtoCOmega}.
Next we estimate the derivatives $\dfrac{\partial^2 \phi}{\partial x_i\partial y_j}(x,Y,X_0)$, $1\leq i\leq n$, $1\leq j\leq n+1$.
We have
\begin{align*}
\dfrac{\partial^2 \phi}{\partial x_i\partial y_j}(x,Y,X_0)
%&=
%\dfrac{1}{1-\kappa^2}
%\,
%\dfrac{-\delta_{ij}\,\sqrt{\dfrac{c(X_0,Y)^{2}}{(1-k^{2})^{2}}-\dfrac{|x-y|^{2}}{1-k^{2}}} -(x_i-y_i)\partial_{y_j}\left( \sqrt{\dfrac{c(X_0,Y)^{2}}{(1-k^{2})^{2}}-\dfrac{|x-y|^{2}}{1-k^{2}}}\right) }{\left(\sqrt{\dfrac{c(X_0,Y)^{2}}{(1-k^{2})^{2}}-\dfrac{|x-y|^{2}}{1-k^{2}}} \right)^2}\\
&=
\dfrac{1}{1-\kappa^2}
\,
\dfrac{-\delta_{ij}\,\sqrt{\dfrac{c(X_0,Y)^{2}}{(1-k^{2})^{2}}-\dfrac{|x-y|^{2}}{1-k^{2}}} -(x_i-y_i)
\left( \dfrac{\dfrac{c(X_0,Y)\,\dfrac{\partial c}{\partial y_j}(X_0,Y)}{(1-\kappa^2)^2}-\dfrac{x_j-y_j}{1-\kappa^2}}{\sqrt{\dfrac{c(X_0,Y)^{2}}{(1-k^{2})^{2}}-\dfrac{|x-y|^{2}}{1-k^{2}}}}\right) 
}{\left(\sqrt{\dfrac{c(X_0,Y)^{2}}{(1-k^{2})^{2}}-\dfrac{|x-y|^{2}}{1-k^{2}}} \right)^2}.
\end{align*}
Now $\dfrac{\partial c(X,Y)}{\partial y_j}=-\dfrac{x_j-y_j}{|X-Y|} -\kappa\,\delta_{j(n+1)}$.
Hence
\begin{align*}
\left|\dfrac{\partial^2 \phi}{\partial x_i\partial y_j}(x,Y,X_0)\right|
%&\leq
%\dfrac{1}{1-\kappa^2}\,\left( \dfrac{1}{\sqrt{\dfrac{c(X_0,Y)^{2}}{(1-k^{2})^{2}}-\dfrac{|x-y|^{2}}{1-k^{2}}}}
%+
%\dfrac{|x-y|\left(\dfrac{(1+\kappa)\,c(X_0,Y)}{(1-k^{2})^{2}}+\dfrac{|x-y|}{1-\kappa^2}\right)}{\left( \sqrt{\dfrac{c(X_0,Y)^{2}}{(1-k^{2})^{2}}-\dfrac{|x-y|^{2}}{1-k^{2}}}\right)^2}
%\right)\\
&\leq
\dfrac{1}{\sqrt{1-\delta^2}\,c(X_0,Y)}
+
\dfrac{|x-y|\left(\dfrac{(1+\kappa)\,c(X_0,Y)}{1-k^{2}}+|x-y|\right)}{\left(\sqrt{1-\delta^2}\,c(X_0,Y)\right)^2}\\
&\leq 
C(\kappa,\delta)\left( \dfrac{1}{c(X_0,Y)} + 1\right)\leq C
\end{align*}
for $x\in \Omega$, $Y\in \mathcal T$ and $X_0\in C_\Omega$ by \eqref{eq:distancefromTtoCOmega}.

Therefore, we obtain
$\left|\dfrac{\partial^2 \phi}{\partial x_i\partial y_j}(x,Y,X_0)\right|\leq C(\delta,K,M,|\Omega|,n)$ for all $x\in \Omega$, $X_0\in C_\Omega$, and $Y\in K$.
Moreover, these estimates also hold for any $X_0$ such that there exists $\bar X\in C_\Omega$ with $c(X_0,Y)=c(\bar X,Y)$. 
%Recall that for any $x\in\Omega$ and any $Y\in \mathcal T$,and any $\bar X\in C_M$, we have $|x-y|^{2}\leq \frac{\delta^{2}c(\bar X,Y)^{2}}{1-k^{2}}$.

Continuing in this manner, we get that the function $\phi(x,Y,X_0)$ is $C^{\infty}$ in all its $3n+2$ variables on $\Omega\times \mathcal T\times C_\Omega$ and $|D^{\alpha}\phi(x,Y,X_0)|\leq C$, for any multi-index $\alpha=(\alpha_1,...,\alpha_{3n+2})$ with a constant $C$ depending only on $\delta, \kappa, M, |\alpha|$ and $|\Omega|$.

\subsection{Assumptions on the target $\Sigma$}\label{subsec:assumptiontargetrefractor}
We will frequently use the following fact:
\begin{align}\label{eq:frecuentfactused} 
\text{if the upper focus $Y$ of the ellipsoid defined by $\phi(x,Y,X_0)$ satisfies}\\
\text{$Y=X_0+s\Lambda(v)$ for some $s>0$ and $v\in \R^n$, then $D_x\phi(x_0,Y,X_0)=v$.}\notag
\end{align}
Here $\Lambda(v)$ is the unit vector given by  \eqref{eq:defofLambdav}.

We assume the convex hull of the target $\Sigma$ is contained in $\mathcal T$, where $\mathcal T$ is given by \eqref{eq:definitionofregionT}.
%Moreover we will assume that the convex hull of $\Sigma$ is contained in $\mathcal T$.
%I THINK WE NEED TO ASSUME OVERALL THAT THE CONVEX ENVELOPE OF $\Sigma$ IS BOUNDED AND CONTAINED IN $\mathcal T$ TO HAVE UNIFORM BOUNDS OF THE DERIVATIVES OF $\phi$.
For each fixed $X\in C_{\Omega}$, we assume each $Y\in \Sigma$ can be represented parametrically with respect to $X$ by the equation $Y=X+s_X(\Lambda)\Lambda$, with $|\Lambda|=1$,  where the function $s_X$ varies with the point $X$, and $s_X(\Lambda)$ is Lipschitz in $\Lambda$ for each $X\in C_\Omega$.
%IS THIS DEPENDENCE CONTINUOUS OR DIFFERENTIABLE?
%IS THIS PARAMETRIZATION LOCAL OR GLOBAL???

\begin{lemma}\label{lm:estimatesofpointsintarget}
For each $X_0\in C_{\Omega}$, there exists a constant $C=C(X_{0})\geq 1$ such that if $\bar Y,\hat Y\in \Sigma$ 
%(HERE WE NEED THE SEGMENT JOINING THESE POINTS BE CONTAINED IN $\mathcal T$?) 
are such that there exist
$\bar v,\hat v\in R^{n}$ and $\bar s,\hat s>0$ with $\bar Y=X_0+\bar s\, \Lambda(\bar v)$, and $\hat Y=X_0+\hat s\, \Lambda(\hat v)$, with $\Lambda$ defined by \eqref{eq:defofLambdav},
then 
\begin{equation}\label{eq:conditionbetweenYandv}
\frac{1}{C}|\bar Y-\hat Y|\leq |\bar v-\hat v|\leq C|\bar Y-\hat Y|.
\end{equation}
This implies that if $\bar Y\neq \hat Y$, are both in $\Sigma$, then the points $\bar Y,\hat Y, X_0$ cannot be aligned, in other words, from each point $X_0$ one can see at most a point in $\Sigma$ on any straight line from $X_0$.\end{lemma}
\begin{proof}
Since the function $s_{X_0}(\Lambda)$ is Lipschitz in $\Lambda$, then the left inequality in \eqref{eq:conditionbetweenYandv}
follows.
Indeed,
\begin{align*}
|\bar Y-\hat Y|&=\left|s_{X_0}(\Lambda(\bar v))\,\Lambda(\bar v)-s_{X_0}(\Lambda(\hat v))\,\Lambda(\hat v)\right|\\
&\leq \left|s_{X_0}(\Lambda(\bar v))\,\Lambda(\bar v)-s_{X_0}(\Lambda(\bar v))\,\Lambda(\hat v)\right| 
+
\left|s_{X_0}(\Lambda(\bar v))\,\Lambda(\hat v)-s_{X_0}(\Lambda(\hat v))\,\Lambda(\hat v)\right|\\
&\leq \left|s_{X_0}(\Lambda(\bar v))\right| \,\left|\Lambda(\bar v)-\Lambda(\hat v)\right| 
+
\left|s_{X_0}(\Lambda(\bar v))-s_{X_0}(\Lambda(\hat v))\right|.
\end{align*}
%TO ESTIMATE THIS WE NEED THE POINTS TO BE CLOSE, BUT $s_{X_0}$ IS LOCALLY LIPSCHITZ???

To show the right inequality in \eqref{eq:conditionbetweenYandv}, from \eqref{eq:frecuentfactused} we can write that 
($x_0\in \Omega$)
\begin{align*}
|\bar v_i-\hat v_i|&=\left|D_i\phi(x_0,\bar Y,X_0)-D_i\phi(x_0,\hat Y,X_0)\right|\\
&=\left|\frac{x_{0_{i}}-\bar y_i}{\bar y_{n+1}-x_{0_{n+1}}-k|X_0-\bar Y|}-\frac{x_{0_{i}}-\hat y_i}{\hat y_{n+1}-x_{0_{n+1}}-k|X_0-\hat Y|}\right|, 
\end{align*}
$1\leq i \leq n$.
%I THINK THIS FOLLOWS FROM THE ESTIMATES FOR THE DERIVATIVES OF $\phi$. HERE IS THE PROOF:

We write
\[
D_i\phi(x_0,\bar Y,X_0)-D_i\phi(x_0,\hat Y,X_0)
=
D_Y(D_i\phi)(x_0,\tilde Y,X_0)\cdot (\bar Y-\hat Y)
\]
with $\tilde Y$ on the segment $[\bar Y,\hat Y]$ joining $\bar Y$ and $\hat Y$. This segment is contained in the convex hull of $\Sigma$, and if this convex hull is bounded and contained in $\mathcal T$, then the desired estimate follows 
from the estimates for the derivatives of $\phi$ proved in Subsection \eqref{subsec:estimatesderivativesofphi}. 

%SECOND PROOF WITHOUT USING THE CONVEX HULL OF $\Sigma\subset \mathcal T$.
%
%
%%From \eqref{eq:upperestimatecwithdelta} 
%We have 
%\begin{equation}\label{eq:inequalityatbarYhatY}
%y_{n+1}-x_{0_{n+1}}-k|X_0-Y|\geq \beta\, c(X_0,Y),
%\end{equation} 
%for $Y=\bar Y,\hat Y$.
%In fact, to obtain this inequality, say at $Y=\bar Y$, we use \eqref{eq:upperestimatecwithdelta}, and to do that we need to verify that $X_0\in E^-(\bar Y,c(X_0,\bar Y))$ which from \eqref{eq:lowerellipsoid} is equivalent to verify that $\bar y_{n+1}-x_{0_{n+1}}-\kappa \,|X_0-\bar Y|\geq 0$.
%Indeed, from \eqref{eq:defofLambdav}
%\[
%\bar y_{n+1}-x_{0_{n+1}}-\kappa \,|X_0-\bar Y|
%=
%\bar y_{n+1}-x_{0_{n+1}}-\kappa \,|\bar s\,\Lambda(\bar v)|
%=
%\bar s(Q(|\bar v|)+\kappa)-\kappa\,\bar s\geq 0.
%\]
%
%
%It then follows from the concavity in $Y$ of  
%the function $(1-\beta\,\kappa)(y_{n+1}-x_{0_{n+1}})-(\kappa+\beta)|X_0-Y|$ that 
%%$y_{n+1}-x_{0_{n+1}}-k|X_0-Y|\geq \beta\, c(X_0,Y)$ 
%\eqref{eq:inequalityatbarYhatY} holds for all $Y\in[\bar Y,\hat Y]$.
%Setting $G_i(Y)=\dfrac{x_{0_{i}}-y_i}{ y_{n+1}-x_{0_{n+1}}-k|X_0- Y|}$, we then obtain from \eqref{eq:estimatesforc(X,Y)} that $|DG_i(Y)|\leq C$ for $Y\in[\bar Y,\hat Y]$ and hence $|G_i(\bar Y)-G_i(\hat Y)|\leq C|\bar Y-\hat Y|$.
\end{proof}

The following two lemmas are a consequence of the inequalities for the derivatives of $\phi$.

\begin{lemma}\label{lm:lipschitzinX0}
Let $\bar X\in C_\Omega$, $Y\in \mathcal T$ and $x_0\in \Omega$. Let $x_{0_{n+1}}=\phi(x_0,Y,\bar X)$ and set $X_0=(x_0,x_{0_{n+1}})$.
Assume that $X_0^{\star}=(x_0,x_{0_{n+1}}-h)\in C_\Omega$. Then there is a constant $C>0$ such that 
\[
0\leq \phi(x,Y,X_0)-\phi(x,Y,X_0^{\star})\leq C\,h,
\] 
for all $x\in\Omega$

\end{lemma}

\begin{proof}
%DOESN'T THIS FOLLOWS AUTOMATICALLY FROM THE BOUNDS FOR THE DERIVATIVES OF $\phi$ WITH RESPECT TO $X_0$???
We have $0\leq \phi(x,Y,X_0)-\phi(x,Y,X_0^{\star})=\dfrac{\partial\phi}{\partial x_{0_{n+1}}}(x,Y,\bar X_0)\,h$ for some $\bar X_0\in [X_0^{\star},X_0]$.
Since $c(\bar X_0,Y)=c(\hat X,Y)$ for some $\hat X\in C_M$, it follows from the estimates in Subsection \eqref{subsec:estimatesderivativesofphi} 
that $\left|\dfrac{\partial\phi}{\partial x_{0_{n+1}}}(x,Y,\bar X_0)\right|\leq C$ for all $x\in\Omega$, where $C$ depends on $\delta, \kappa, M$ and $|\Omega|$.
%HERE WE DO NOT KNOW THAT THE ESTIMATE OF $\left|\frac{\partial c(x,\xi,Y)}{\partial x_{n+1}}\right|$ HOLDS FOR POINTS $X$ NOT ON THE LOWER PART OF THE ELLIPSOID???
%WE PROBABLY NEED TO HAVE \eqref{eq:upperestimateofpartialcxn+1} FOR MORE POINTS?
\end{proof}

\begin{lemma}\label{lm:estimateslipinYandx}

There exists a constant $C>0$ such that for all $Y,\bar Y\in \mathcal T$ with the straight segment $[Y,\bar Y]\subset \mathcal T$, and $X_0 \in C_{\Omega}$, we have 
\[
|\phi(x,Y,X_0)-\phi(x,\bar Y,X_0)|\leq C\,|x-x_0|\,|Y-\bar Y|
\] 
for all $x\in \Omega$, which is assumed convex.
\end{lemma}
\begin{proof}

Since $\phi(x_0,Y,X_0)=x_{0_{n+1}}$ for all $Y\in \Sigma$, it follows that $\dfrac{\partial\phi(x_0,Y,X_0)}{\partial y_{j}}=0$. We then write 
\begin{align*}
\phi(x,Y,X_0)-\phi(x,\bar Y,X_0)&=\sum_{j=1}^{n+1}\dfrac{\partial\phi(x,\xi,X_0)}{\partial y_{j}}(Y_j-\bar Y_j), 
\quad \text{for some $\xi\in [Y,\bar Y]$}\\ 
&=\sum_{j=1}^{n+1}\left(\frac{\partial\phi(x,\xi,X_0)}{\partial y_{j}}-\frac{\partial\phi(x_0,\xi,X_0)}{\partial y_{j}}\right)(Y_j-\bar Y_j)\\
&=\sum_{j=1}^{n+1}\sum_{i=1}^{n}\frac{\partial^{2}\phi(\zeta_i,\xi,X_0)}{\partial x_{i}\partial y_{j}}(x_{i}-x_{0_{i}})(Y_j-\bar Y_j)
\end{align*}
 for some $\zeta_i\in [x_0,x]$.
%The segment $[Y,\bar Y]$ is contained in the convex hull of $\Sigma$ which we assume it bounded and contained in $\mathcal T$. Also we assume $\Omega$ convex so $[x_0,x]\subset \Omega$.
Then the lemma follows from the estimates of the derivatives of $\phi$ proved in Subsection \eqref{subsec:estimatesderivativesofphi}.
\end{proof}

\setcounter{equation}{0}
\section{Regularity hypothesis on the target set $\Sigma$}\label{sec:regularityhypothesesonthetarget}

Given $\bar Y,\hat Y\in \mathcal T$ and $X_0\in C_{\Omega}$, let $\bar v=D_x\phi(x_0,\bar Y,X_0)$, $\hat v=D_x\phi(x_0,\hat Y,X_0)$,
 and $v(\lambda)=(1-\lambda)\bar v +\lambda \hat v$, with $\lambda \in [0,1]$. We consider the set of points 
 \[
 C(X_0,\bar Y,\hat Y)=\{X_0+s\,\Lambda(v(\lambda)):s>0,\lambda\in [0,1]\}, 
 \]
where $\Lambda(v(\lambda))$ is defined by \eqref{eq:defofLambdav}.
This set is a two dimensional  wedge-shaped surface, in general non-planar, containing all rays having directions $\Lambda(v(\lambda))$, $0\leq \lambda \leq 1$,
%$\Lambda(\bar v)$ and $\Lambda(\hat v)$
emanating from $X_0$. The curve describing the tip of the vector $\Lambda(v(\lambda))$, pictured in Figure \ref{fig:picofcone}, is not contained in the plane generated  by the rays with directions $\Lambda(\bar v)$ and $\Lambda(\hat v)$. If $Y\in C(X_0,\bar Y,\hat Y)$, then $X_0\in E^-(Y,c(X_0,Y))$. Hence, if $Y(\lambda)=X_0+s\,\Lambda(v(\lambda))\in C(X_0,\bar Y,\hat Y)$, then by \eqref{eq:frecuentfactused} $D_x \phi(x_0,Y(\lambda),X_0)=v(\lambda)$. 
%MAKE A PICTURE.
\begin{figure}[htp]
\begin{center}
   \includegraphics[width=2in]{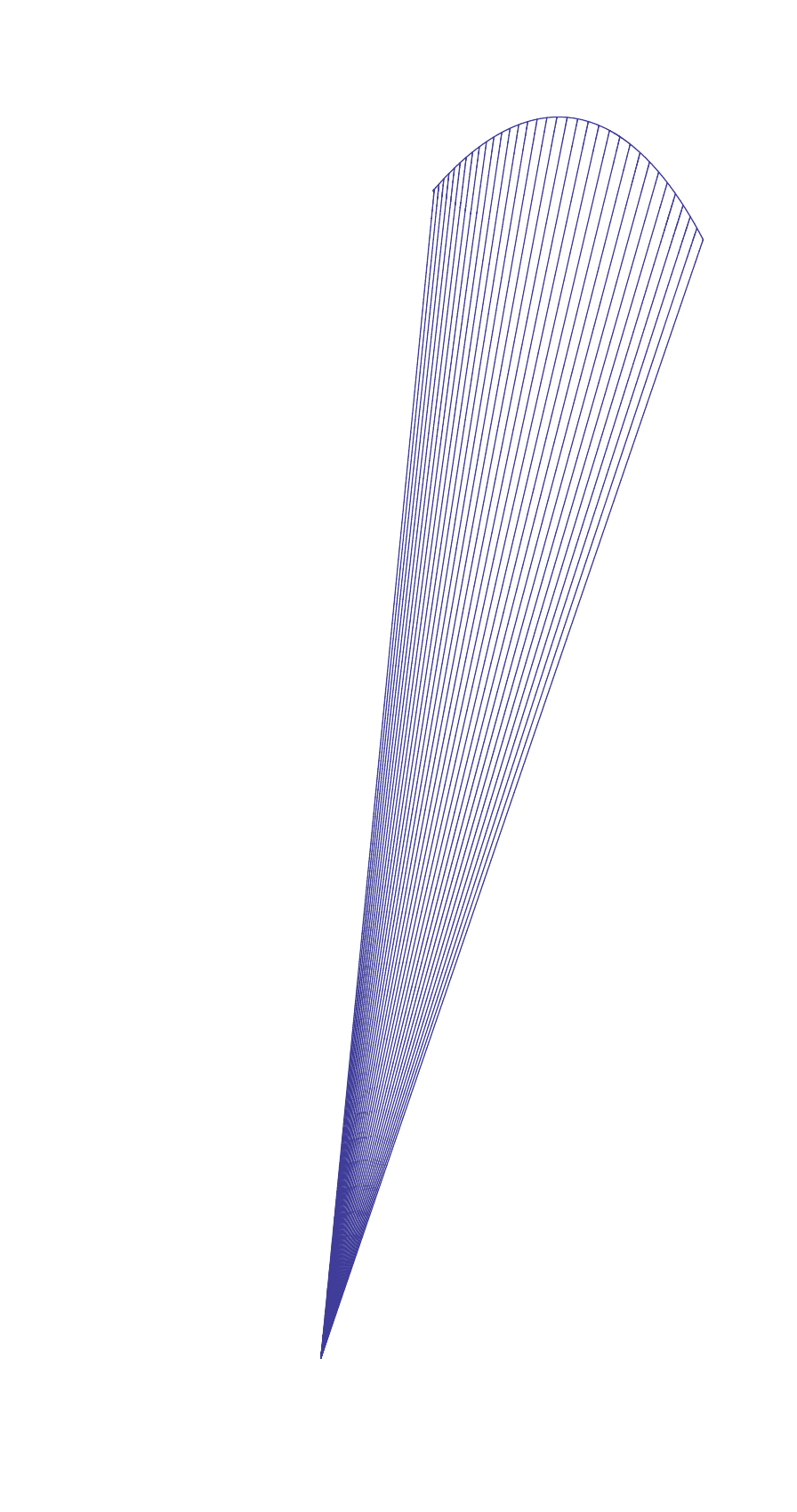}
    %\subfigure[$|X|+2/3 |X-P|=1.7$, $P=(2,0)$]{\label{fig:edge-b}\includegraphics[width=2.9in]{snelllaw.pdf}} 
%    \subfigure[After Sobel edge detection]{\label{fig:edge-c}\includegraphics[scale=1]{SeveralCartesianOvalskappa<1.pdf}}
\end{center}
  \caption{The edges of the wedge are $\Lambda(\bar v)$ and $\Lambda(\hat v)$, with $\bar v=(1,0)$, $\hat v=(-2,1)$, $n_1=1$ , $n_2=3/2$}
  \label{fig:picofcone}
\end{figure}

For $\bar Y,\hat Y\in \mathcal T$, and $X_0\in C_\Omega$,  
let us define 
\[
[\bar Y,\hat Y]_{X_0}:=\Sigma \cap C(X_0,\bar Y,\hat Y).
\]
Notice that by \eqref{eq:conditionbetweenYandv} each ray $X_0+s\,\Lambda(v(\lambda))\in C(X_0,\bar Y,\hat Y)$ intersects $\Sigma$ in at most one point for each $\lambda$.
Points in $[\bar Y,\hat Y]_{X_0}$ have the form 
\[
Y(\lambda)=X_0+s_{X_0}(\Lambda(v(\lambda)))\,\Lambda(v(\lambda)), \qquad 0\leq \lambda\leq 1.
\]
%Notice that if $Y(\lambda)\in [\bar Y,\hat Y]_{X_0}$, then $D\phi(x_0,Y(\lambda) ,X_0)=v(\lambda)$. 

We will introduce the following local definition on the target $\Sigma$.
%\begin{definition}\label{def:localtarget}
%If $X_0\in C_\Omega$ we say that the {\it target $\Sigma$ is regular at $X_0$} if 
%there exists a neighborhood $U_{X_0}$ and positive constants $C_1,C_2$, depending on $U_{X_0}$, such that for all $\bar Y,\hat Y\in \Sigma$ and $Z=(z,z_{n+1})\in U_{X_0}$
%we have
%\begin{equation}\label{eq:convexityconditiononthetarget}
%\phi(x,Y_{\bar X}(\lambda),Z)\geq(1-\lambda)\phi(x,\bar Y,Z)+\lambda\phi(x,\hat Y,Z)+C_1\,|\bar Y-\hat Y|^{2}|x-z|^{2}\lambda(1-\lambda)
%\end{equation}
%for all $x\in\Omega$, $0\leq \lambda \leq 1$ with $|x-z|\leq C_2$, and $Y_{Z}(\lambda)=Z+s_{Z}(\Lambda(v(\lambda)))\,\Lambda(v(\lambda))$. 
%%SIMILAR TO A STRONG, CHECK MCCANN-FIGALLI-RECENT PAPER.
%%Where $C$ stands for possibly different structural constants.
%%THE CONSTANTS C DEPEND ON THE POINTS OR ARE UNIFORM???
%Here $\bar v=D_x\phi(z,\bar Y,Z)$, $\hat v=D_x\phi(z,\hat Y,Z)$,
% and $v(\lambda)=(1-\lambda)\bar v +\lambda \hat v$.
% \end{definition}
% If this is satisfied we say the target $\Sigma$ satisfies 

%WE CAN ALSO MAKE THE FOLLOWING ALTERNATIVE DEFINITION.
\begin{definition}\label{def:localtargetbis}
If $X_0\in C_\Omega$ we say that the {\it target $\Sigma$ is regular from $X_0$} if 
there exists a neighborhood $U_{X_0}$ and positive constants $C_1,C_2$, depending on $U_{X_0}$, such that for all $\bar Y,\hat Y\in \Sigma$ and $Z=(z,z_{n+1})\in U_{X_0}$
we have
\begin{equation}\label{eq:convexityconditiononthetargetbisbis}
\phi(x,Y_Z(\lambda),Z)\geq\min\left\{\phi(x,\bar Y,Z),\phi(x,\hat Y,Z)\right\}+C_1\,|\bar Y-\hat Y|^{2}|x-z|^{2}\end{equation}
for all $x\in \Omega$ with $|x-z|\leq C_2$, $1/4\leq \lambda \leq 3/4$, and $Y_{Z}(\lambda)=Z+s_{Z}(\Lambda(v(\lambda)))\,\Lambda(v(\lambda))$. 
%SIMILAR TO A STRONG, CHECK MCCANN-FIGALLI-RECENT PAPER.
%Where $C$ stands for possibly different structural constants.
%THE CONSTANTS C DEPEND ON THE POINTS OR ARE UNIFORM???
Here $\bar v=D_x\phi(z,\bar Y,Z)$, $\hat v=D_x\phi(z,\hat Y,Z)$,
 and $v(\lambda)=(1-\lambda)\bar v +\lambda \hat v$.
 
% FOR THE APPLICATION WE ONLY NEED TO HAVE IN THE DEFINITION $1/4\leq \lambda \leq 3/4$.
% THE TERM $\lambda (1-\lambda)$ IN \eqref{eq:convexityconditiononthetargetbisbis} CAN BE OMITTED.
% CHECK THE USE IN THE PROOF OF THE THEOREMS.
 \end{definition}

We introduce below the differential condition \eqref{eq:AWP}, similar in form to condition (A3) of Ma, Trudinger and Wang \cite{MaTrudingerWang:regularityofpotentials}. Assuming that the function $s_{X_0}$ in the parametrization of the target is $C^2$, and the set $[\bar Y,\hat Y]_{X_0}$ is a curve for each $X_0$,
we prove in the following theorem that 
\eqref{eq:AWP} is equivalent to \eqref{eq:convexityconditiononthetargetbisbis}, see also Remark \ref{rmk:localequivalenceperpvectors} for the local case.
Theorem \ref{thm:awforperpendicularvectors} is similar in form to a result of  
Loeper, \cite[Proposition 5.1]{loeper:actapaper}.

\begin{theorem}\label{thm:awforperpendicularvectors}
Suppose that there exists a constant $C$ such that for all $\xi$ and $\eta$, perpendicular vectors in $R^{n}$, and for  $X_0\in C_{\Omega}$ and for $Y_0\in\Sigma$, we have 
\begin{equation}\label{eq:AWP}
\dfrac{d^{2}}{d\epsilon^{2}}\left. \left\langle D_x^{2}\phi(x_0,Y_\epsilon,X_0)\eta,\eta\right\rangle \right|_{\epsilon=0}\leq -C|\xi|^{2}|\eta|^{2},
\end{equation}
where, as before $v_0=D\phi(x_0,Y_0,X_0)$ and $Y_\epsilon=X_0+s(\Lambda(v_0+\epsilon\xi))\Lambda(v_0+\epsilon\xi)$.

Then there exist structural constants $\sigma$ and $C$ such that for $\bar Y,\hat Y\in \Sigma$ and $X_0\in C_{\Omega}$ we have for $\lambda\in[1/4,3/4]$ and for $|x-x_0|\leq \sigma$ that
\begin{equation}\label{eq:minconditiontarget}
\phi(x,Y(\lambda),X_0)\geq \min\{\phi(x,\bar Y,X_0),\phi(x,\hat Y,X_0)\}+C|\bar Y-\hat Y|^{2}|x-x_0|^{2}.
\end{equation}

Conversely, \eqref{eq:minconditiontarget} implies \eqref{eq:AWP}.
\end{theorem}
\begin{proof}

Let $\bar Y,\hat Y\in\Sigma$ and $X_0\in C_{\Omega}$. Let $v(\lambda)=(1-\lambda)D\phi(x_0,\bar Y,X_0)+\lambda D\phi(x_0,\hat Y,X_0)$, $\xi=D\phi(x_0,\hat Y,X_0)-D\phi(x_0,\bar Y,X_0)=\hat v-\bar v$, and $Y(\lambda)=X_0+s_{X_0}(\Lambda(v(\lambda)))\Lambda(v(\lambda))$.
Notice that for $\eta$ perpendicular to $\xi$ we have 
\begin{equation}\label{eq:AWPbis}
\dfrac{d^{2}}{d\lambda^{2}}\left\langle D_x^{2}\phi(x_0,Y(\lambda),X_0)\eta,\eta\right\rangle \leq -C|\xi|^{2}|\eta|^{2},
\end{equation} 
for all $\lambda\in[0,1]$.
Indeed, fix $\lambda\in (0,1)$ and $\xi,\eta\in R^{n}$, $\xi\perp \eta$. 
Applying \eqref{eq:AWP} with $Y_0
\rightsquigarrow Y(\lambda)$, $v_0\rightsquigarrow v(\lambda):=\bar v+\lambda\xi$ (notice that $\bar v+\lambda\xi=D_x\phi(x_0,Y(\lambda),X_0)$ by \eqref{eq:frecuentfactused}),
and $Y_\epsilon \rightsquigarrow X_0+s_{X_0}(\Lambda(v(\lambda)+\epsilon\xi))\Lambda(v(\lambda)+\epsilon\xi)$,
we obtain
$\dfrac{d^{2}}{d\epsilon^{2}}\langle D^{2}\phi(x_0, Y_\epsilon,X_0)\eta,\eta\rangle|_{\epsilon=0}\leq -C|\xi|^{2}|\eta|^{2}$.
Since $Y_\epsilon=Y(\epsilon+\lambda)$, \eqref{eq:AWPbis} follows.

For $x\in R^{n}$, let $x^{\prime}$ denote the orthogonal projection of $x$ on the hyperplane through $x_0$ and normal $\xi$, so $x^{\prime}-x_0$ is perpendicular to $\xi$.
We let $\eta=x^{\prime}-x_0$.

We will first show that there exist positive constants $C_1$ and $C_2$ such that 
\begin{align}\label{eq:inequalityofhessians}
&\left\langle D_x^{2}\phi(x_0,Y(\lambda),X_0)(x-x_0),x-x_0\right\rangle\\
&\geq \left\langle ((1-\lambda)D_x^{2}\phi(x_0,\bar Y,X_0)+\lambda D_x^{2}\phi(x_0,\hat Y,X_0))(x-x_0),x-x_0\right\rangle \notag\\
&\qquad +\lambda(1-\lambda)|\xi|^{2}(C_1|x-x_0|^{2}-C_2|x-x^{\prime}|^{2}),\notag
\end{align}
for all $x\in \Omega$.
In fact, fix $x\in \Omega$ and let $f(\lambda)=-\langle D_x^{2}\phi(x_0,Y(\lambda),X_0)(x^{\prime}-x_0),x^{\prime}-x_0\rangle$. From \eqref{eq:AWPbis}, we have $f^{\prime\prime}(\lambda)\geq C|x^{\prime}-x_0|^{2}|\xi|^{2}$.
%and therefore $f(\lambda)\leq (1-\lambda)f(0)+\lambda f(1)-C\lambda(1-\lambda)|x^{\prime}-x_0|^{2}|\xi|^{2}$.
We claim that  $f(\lambda)\leq (1-\lambda)f(0)+\lambda f(1)-C|x'-x_0|^{2}|\xi|^{2}\lambda(1-\lambda)$ for all $\lambda\in [0,1]$. To prove this claim, 
fix $\bar\lambda\in (0,1)$. 
By Taylor's theorem, we have that $f(\lambda)\geq  f(\bar\lambda)+f^{\prime}(\bar\lambda)(\lambda-\bar\lambda)+C/2\,|x'-x_0|^{2}|\xi|^{2}(\lambda-\bar\lambda)^{2}$ for all $\lambda\in [0,1]$. Applying this inequality, first for $\lambda=0$ and then for $\lambda=1$, multiplying the first by $1-\bar\lambda$ and the second by $\bar \lambda$, and then adding yields $(1-\bar\lambda)f(0)+\bar\lambda f(1)\geq f(\bar\lambda)+C|x'-x_0|^{2}|\xi|^{2}\bar \lambda(1-\bar \lambda)$ which proves the claim.

Let $g(\lambda)=\langle D_x^{2}\phi(x_0,Y(\lambda),X_0)(x^{\prime}-x_0),x^{\prime}-x_0\rangle-\langle D_x^{2}\phi(x_0,Y(\lambda),X_0)(x-x_0),x-x_0\rangle$.
Since $s_{X_0}$ is $C^2$, we have that $\left|\dfrac{d^2\, D_{ij}\phi(x_0,Y(\lambda),X_0)}{d\lambda^2}\right|\leq C\,|\xi|^2$ and therefore $|g^{\prime\prime}(\lambda)|\leq C|\xi|^{2}|x-x^{\prime}||x-x_0|$, since $|x'-x_0|\leq |x-x_0|$.
Hence $g(\lambda)\leq (1-\lambda)g(0)+\lambda g(1)+C\lambda(1-\lambda)|\xi|^{2}|x-x^{\prime}||x-x_0|$, for $0\leq \lambda \leq 1$.
Therefore we get 
\begin{align*}
&-\left\langle D_x^{2}\phi(x_0,Y(\lambda),X_0)(x-x_0),x-x_0\right\rangle
=f(\lambda)+g(\lambda)\\
&\leq (1-\lambda)(f(0)+g(0))+\lambda(f(1)+g(1))+\lambda(1-\lambda)|\xi|^{2}\left(C_1\,|x-x^{\prime}||x-x_0|-C_2\,|x^{\prime}-x_0|^{2}\right)\\
&\leq -\left\langle \left((1-\lambda)D_x^{2}\phi(x_0,\bar Y,X_0)+\lambda D_x^{2}\phi(x_0,\hat Y,X_0)\right)(x-x_0),x-x_0\right\rangle\\
&\qquad +\lambda(1-\lambda)|\xi|^{2}(C_1'|x-x^{\prime}|^{2}-C_2'|x-x_0|^{2}),
\end{align*} 
since $|x'-x_0|^2=|x-x_0|^2-|x-x'|^2$, with $C_1,C_2$ positive structural constants.
This finishes the proof of \eqref{eq:inequalityofhessians}.

We now prove \eqref{eq:minconditiontarget}.
Let $\lambda^{\prime}\in [0,1]$ to be chosen later, and let $\lambda\in[1/4,3/4]$. We have
\begin{align*}
&\min\{\phi(x,\bar Y,X_0),\phi(x,\hat Y,X_0)\}-\phi(x,Y(\lambda),X_0)\\
&\leq (1-\lambda^{\prime})\phi(x,\bar Y,X_0)+\lambda^{\prime}\phi(x,\hat Y,X_0)-\phi(x,Y(\lambda),X_0)\\
&=\left\langle ((1-\lambda^{\prime})D_x\phi(x_0,\bar Y,X_0)+\lambda^{\prime} D_x\phi(x_0,\hat Y,X_0))-D_x\phi(x_0,Y(\lambda),X_0),x-x_0\right\rangle\\
&+\frac{1}{2}\left\langle ((1-\lambda^{\prime})D_x^{2}\phi(x_0,\bar Y,X_0)+\lambda^{\prime} D_x^{2}\phi(x_0,\hat Y,X_0)-D_x^{2}\phi(x_0,Y(\lambda),X_0))(x-x_0),x-x_0\right\rangle\\
&+\frac{1}{6}\sum_{i,j,k=1}^{n}\left((1-\lambda^{\prime})D_{i,j,k}\phi(\tau,\bar Y,X_0)+\lambda^{\prime} D_{i,j,k}\phi(\tau,\hat Y,X_0)-D_{i,j,k}\phi(\tau,Y(\lambda),X_0)\right)(x_{i}-x_{0_i})(x_{j}-x_{0_j})(x_{k}-x_{0_k})
\end{align*} for some $\tau\in [x_0,x]$. Using \eqref{eq:inequalityofhessians} in $\lambda'$ we get that the last sum is less than or equal to 
\begin{align*}
&\left\langle (D\phi(x_0,Y(\lambda^{\prime}),X_0)-D\phi(x_0,Y(\lambda),X_0)),x-x_0\right\rangle\\
&+\dfrac12 \left\langle (D_x^{2}\phi(x_0,Y(\lambda^{\prime}),X_0)-D_x^{2}\phi(x_0,Y(\lambda),X_0))(x-x_0),x-x_0\right\rangle\\
&\qquad -C_1\lambda^{\prime}(1-\lambda^{\prime})|\xi|^{2}|x-x_0|^{2}+C_2\lambda^{\prime}(1-\lambda^{\prime})|\xi|^{2}|x-x^{\prime}|^{2}\\
&+\frac16 \sum_{i,j,k=1}^{n}\left((1-\lambda^{\prime})D_{i,j,k}\phi(\tau,\bar Y,X_0)+\lambda^{\prime} D_{i,j,k}\phi(\tau,\hat Y,X_0)-D_{i,j,k}\phi(\tau,Y(\lambda),X_0)\right)(x_{i}-x_{0_i})(x_{j}-x_{0_j})(x_{k}-x_{0_k}).
\end{align*} 

Notice that $\langle (D\phi(x_0,Y(\lambda^{\prime}),X_0)-D\phi(x_0,Y(\lambda),X_0)),x-x_0\rangle=(\lambda^{\prime}-\lambda)\langle\xi,x-x_0\rangle=(\lambda^{\prime}-\lambda)\langle\xi,x-x^{\prime}\rangle$ by orthogonality. 
Hence we get that
\begin{align*}&\min\{\phi(x,\bar Y,X_0),\phi(x,\hat Y,X_0)\}-\phi(x,Y(\lambda),X_0)\\
&\leq
-C_1\lambda(1-\lambda)|\xi|^{2}|x-x_0|^{2}+(\lambda^{\prime}-\lambda)\langle\xi,x-x^{\prime}\rangle\\
&+C_2\lambda(1-\lambda)|\xi|^{2}|x-x^{\prime}|^{2}+C_1|\xi|^{2}|x-x_0|^{2}(\lambda(1-\lambda)-\lambda^{\prime}(1-\lambda^{\prime}))\\
&+C_2|\xi|^{2}|x-x^{\prime}|^{2}(\lambda^{\prime}(1-\lambda^{\prime})-\lambda(1-\lambda))\\
&+\left\langle (D_x^{2}\phi(x_0,Y(\lambda^{\prime}),X_0)-D_x^{2}\phi(x_0,Y(\lambda),X_0))(x-x_0),x-x_0\right\rangle\\
&+\sum_{i,j,k=1}^{n}((1-\lambda^{\prime})D_{i,j,k}\phi(\tau,\bar Y,X_0)+\lambda^{\prime} D_{i,j,k}\phi(\tau,\hat Y,X_0)-D_{i,j,k}\phi(\tau,Y(\lambda),X_0))(x_{i}-x_{0_i})(x_{j}-x_{0_j})(x_{k}-x_{0_k}).
\end{align*} 
We will now choose $\lambda^{\prime}$ and estimate each of the terms.
We have that 
\begin{align*}
A&:=(\lambda^{\prime}-\lambda)\langle\xi,x-x^{\prime}\rangle+C_2\lambda(1-\lambda)|\xi|^{2}|x-x^{\prime}|^{2}
=\left\langle x-x^{\prime},(\lambda^{\prime}-\lambda)\xi+C_2\lambda(1-\lambda)|\xi|^{2}(x-x^{\prime})\right\rangle\\
&=\left\langle x-x^{\prime},|\xi|\left((\lambda^{\prime}-\lambda)\frac{\xi}{|\xi|}+C_2\lambda(1-\lambda)|\xi|(x-x^{\prime})\right)\right\rangle.
\end{align*} 
Notice that either $\dfrac{x-x^{\prime}}{|x-x^{\prime}|}=\dfrac{\xi}{|\xi|}$ or $\dfrac{x-x^{\prime}}{|x-x^{\prime}|}=-\dfrac{\xi}{|\xi|}$.
If the $+$ sign holds, we choose $\lambda^{\prime}=\lambda-C_2\lambda(1-\lambda)|\xi||x-x^{\prime}|$, 
and if the $-$ sign holds, we choose $\lambda^{\prime}=\lambda+C_2\lambda(1-\lambda)|\xi||x-x^{\prime}|$, so in either case $A=0$.
Notice that since $1/4\leq \lambda \leq 3/4$, we have from Lemma \ref{lm:estimatesofpointsintarget} that 
$|\lambda^{\prime}-\lambda|\leq \frac{C_2}{16}|\xi||x-x^{\prime}|\leq \frac{C_3}{16}|\bar Y-\hat Y||x-x_0|\leq C|x-x_0|\leq \frac{1}{4}$, if $|x-x_0|<1/4C$.  Hence $\lambda^{\prime}\in [0,1]$. 
%and with this choice $(\lambda^{\prime}-\lambda)\langle\xi,x-x^{\prime}\rangle+C_2\lambda(1-\lambda)|\xi|^{2}|x-x^{\prime}|^{2}=0$.

We next estimate the remaining terms.
We have $C_1|\xi|^{2}|x-x_0|^{2}(\lambda(1-\lambda)-\lambda^{\prime}(1-\lambda^{\prime}))\leq 3C_1|\xi|^{2}|x-x_0|^{2}|\lambda^{\prime}-\lambda|\leq C|\bar Y-\hat Y|^{3}|x-x_0|^{3}$, again by Lemma \ref{lm:estimatesofpointsintarget}.
Similarly,
$C_2|\xi|^{2}|x-x^{\prime}|^{2}(\lambda^{\prime}(1-\lambda^{\prime})-\lambda(1-\lambda))\leq C|\bar Y-\hat Y|^{3}|x-x_0|^{3}$. 

We also have from the estimates at the end of Subsection \ref{subsec:estimatesderivativesofphi} that
\begin{align*}
&\left\langle (D_x^{2}\phi(x_0,Y(\lambda^{\prime}),X_0)-D_x^{2}\phi(x_0,Y(\lambda),X_0))(x-x_0),x-x_0\right\rangle\\&\leq C|Y(\lambda^{\prime})-Y(\lambda)|\,|x-x_0|^{2}
\leq C|v(\lambda^{\prime})-v(\lambda)||x-x_0|^{2}\\
&=|\lambda^{\prime}-\lambda|\,|\bar v-\hat v|\,|x-x_0|^{2}\leq C|\lambda^{\prime}-\lambda|\,|\bar Y-\hat Y|\,|x-x_0|^{2}\leq C|\bar Y-\hat Y|^{2}|x-x_0|^{3}.
\end{align*} 
 
To estimate the cubic form, let $h(\lambda)=D_{i,j,k}\phi(\tau,Y(\lambda),X_0)$. 
As in the estimate of $g''$ above, we have $|h^{\prime\prime}(\lambda)|\leq C|\xi|^{2}$.   
We have 
\begin{align*}
&\left|((1-\lambda^{\prime})D_{i,j,k}\phi(\tau,\bar Y,X_0)+\lambda^{\prime} D_{i,j,k}\phi(\tau,\hat Y,X_0)-D_{i,j,k}\phi(\tau,Y(\lambda),X_0))(x_{i}-x_{0_i})(x_{j}-x_{0_j})(x_{k}-x_{0_k})\right|\\
&\leq C|x-x_0|^{3}|(1-\lambda^{\prime})h(0)+\lambda^{\prime}h(1)-h(\lambda)|\\
&\leq C|x-x_0|^{3}|(1-\lambda^{\prime})h(0)+\lambda^{\prime}h(1)-h(\lambda^{\prime})|+C|x-x_0|^{3}|h(\lambda^{\prime})-h(\lambda)|\\
&\leq C|x-x_0|^{3}|\xi|^{2}+C|x-x_0|^{3}|\lambda^{\prime}-\lambda|\,|\xi|
\leq C|x-x_0|^{3}|\bar Y-\hat Y|^{2} +C|x-x_0|^{4}|\bar Y-\hat Y|^2.
\end{align*}

Combining all these estimates we obtain 
\begin{align*}
&\min\{\phi(x,\bar Y,X_0),\phi(x,\hat Y,X_0)\}-\phi(x,Y(\lambda),X_0)\\
&\leq -C_1^{\prime}|\bar Y-\hat Y|^{2}|x-x_0|^{2}+C(|\bar Y-\hat Y|^{3}|x-x_0|^{3}+|\bar Y-\hat Y|^{2}|x-x_0|^{3}+|\bar Y-\hat Y|^{2}|x-x_0|^{4}),
\end{align*} 
where $C_1'$ and $C$ are structural constants.  To obtain our desired estimate we write  
\begin{align*}
&C\left(|\bar Y-\hat Y|^{3}|x-x_0|^{3}+|\bar Y-\hat Y|^{2}|x-x_0|^{3}+|\bar Y-\hat Y|^{2}|x-x_0|^{4}\right)\\
&=C|x-x_0|^{2}|\bar Y-\hat Y|^{2}\left(|\bar Y-\hat Y||x-x_0|+|x-x_0|+|x-x_0|^{2}\right)\\
&\leq \frac{C_1^{\prime}}{2}|x-x_0|^{2}|\bar Y-\hat Y|^{2},
\end{align*} 
provided we choose $|x-x_0|\leq C_2'$, with $C_2'$ sufficiently small depending only on the structure.
Taking $\sigma=\min\{1/4C, C_2'\}$, the first part of the lemma follows.

We finally show that \eqref{eq:minconditiontarget} implies \eqref{eq:AWP}.
Fix $X_0\in C_{\Omega}$, $Y_0\in\Sigma$ and let $\xi$ and $\eta$ be perpendicular vectors in $R^{n}$. Set $v_0=D\phi(x_0,Y_0,X_0)$ and hence we can write $Y_0=X_0+s_{X_0}(\Lambda(v_0))\Lambda(v_0)$. Let $Y_{\epsilon}\in\Sigma$ be given by $Y_{\epsilon}=X_0+s(\Lambda(v_0+\epsilon\xi))\Lambda(v_0+\epsilon\xi))$ and notice that $Y_0\in [Y_{-\epsilon},Y_{\epsilon}]_{X_0}$ for the value $\lambda=\frac{1}{2}$ and for all $\epsilon>0$. Therefore from \eqref{eq:minconditiontarget} we have $\phi(x,Y_0,X_0)\geq \min\{\phi(x,Y_{-\epsilon},X_0),\phi(x,Y_{\epsilon},X_0)\}+C_1\epsilon^{2}|\xi|^{2}|x-x_0|^{2}$ for $|x-x_0|\leq C_2$.

Let $S_{\epsilon}=\{x\in\Omega:\phi(x,Y_{-\epsilon},X_0)=\phi(x,Y_{\epsilon},X_0)\}$. Notice that $D\phi(x_0,Y_{\epsilon},X_0)-D\phi(x_0,Y_{-\epsilon},X_0)=2\epsilon\xi$ is a normal vector to $S_{\epsilon}$ at $x_0$.

Let $\gamma$ be a curve contained in $S_{\epsilon}$ such that $\gamma(0)=x_0$ and $\gamma^{\prime}(0)=\eta$. We then have that $\phi(\gamma(t),Y_0,X_0)\geq \frac{1}{2}\phi(\gamma(t),Y_{-\epsilon},X_0)+\frac{1}{2}\phi(\gamma(t),Y_{\epsilon},X_0)\}+C_1\epsilon^{2}|\xi|^{2}|\gamma(t)-x_0|^{2}$ for all $|t|$ small enough.

Let $g(t)=\phi(\gamma(t),Y_0,X_0)-\frac{1}{2}\phi(\gamma(t),Y_{-\epsilon},X_0)-\frac{1}{2}\phi(\gamma(t),Y_{\epsilon},X_0)\}-C_1\epsilon^{2}|\xi|^{2}|\gamma(t)-x_0|^{2}$.

We have $g^{\prime}(0)=0$ and $g^{\prime\prime}(0)\geq 0$ and that is 
\[
\left\langle ( D_x^{2}\phi(x_0,Y_0,X_0)-\frac{1}{2}D_x^{2}\phi(x_0,Y_{-\epsilon},X_0)-\frac{1}{2}D_x^{2}\phi(x_0,Y_{\epsilon},X_0)-2C_1\epsilon^{2}|\xi|^{2}\,Id)\eta,\eta\right\rangle\geq 0
\]
and this inequality holds for all $\epsilon$ small enough.
Letting $h(\epsilon)=\langle (D_x^{2}\phi(x_0,Y_{\epsilon},X_0))\eta,\eta \rangle$, we get that $h(0)-\frac12 h(-\epsilon)-\frac12 h(\epsilon)\geq 2C_1\epsilon^{2}|\xi|^{2}|\eta|^{2}$, we get that $h^{\prime\prime}(0)\leq -4C_1|\xi|^{2}|\eta|^{2}$. 
Therefore we obtain that
 $\dfrac{d^{2}}{d\epsilon^{2}}\left. \left\langle D_x^{2}\phi(x_0,Y_\epsilon,X_0)\eta,\eta\right\rangle \right|_{\epsilon=0}\leq -C|\xi|^{2}|\eta|^{2}$.
 \end{proof}
 
 \begin{remark}\label{rmk:localequivalenceperpvectors}\rm
 
 The local version of \eqref{eq:AWP} can be stated as follows: The target $\Sigma$ is regular from $X_0\in C_{\Omega}$ if there exists a neighborhood $U_{X_0}$ and a constant $C$ depending on $X_0$ such that for all $Y_0\in\Sigma$ and for all $Z\in U_{X_0}$ and for all vectors $\xi$ and $\eta$ such that $\xi\perp \eta$ we have 
\begin{equation}\label{eq:AWPL}
\dfrac{d^{2}}{d\epsilon^{2}}\left. \left\langle D_x^{2}\phi(z,Y_\epsilon,Z)\eta,\eta\right\rangle \right|_{\epsilon=0}\leq -C|\xi|^{2}|\eta|^{2},
\end{equation}
where $Y_{\epsilon}=Z+s(\Lambda(v+\epsilon\xi))\Lambda(v+\epsilon\xi)$ and $v=D\phi(z,Y_0,Z)$.

Following the proof of Theorem \ref{thm:awforperpendicularvectors}, one can show that \eqref{eq:AWPL} is equivalent to \eqref{eq:convexityconditiononthetargetbisbis}.
% for $\lambda\in [1/4,3/4]$.
\end{remark}

\section{Local and global refractors}\label{subsec:localandglobalrefractors}

\subsection{Refractors}
Let $u:\Omega\to [0,M]$,
%with $\Omega$ convex, 
and assume that the convex hull of $\Sigma\subset \mathcal T$, and $\Omega$ is connected. Given $x_0\in \Omega$, set $X_0=(x_0,u(x_0))$.
We define 
\begin{equation}\label{eq:definitionofrefractormapping}
F_u(x_0)=\{Y\in \Sigma: \text{$u(x)\leq \phi(x,Y,X_0)$ for all $x\in\Omega$}\}.
\end{equation}
The function $u$ is a {\it parallel refractor} if $F_u(x_0)\neq \emptyset$ for all $x_0\in \Omega$.

We notice that, from the estimates of the derivatives $\partial_{x_i}\phi$ from Subsection 
\ref{subsec:estimatesderivativesofphi}, any refractor is a Lipschitz function in $\Omega$ with a Lipschitz constant 
depending only on $\delta$ in \eqref{eq:definitionofregionT}. 
%PROVE THIS FACT, THE LIP CONSTANT SHOULD DEPEND ON $\delta$.

%We make the following hypothesis on the measures.

%\subsection

Suppose that $u$ is a parallel refractor in $\Omega$. In general, a local supporting ellipsoid might not support the refractor in all of $\Omega$.
%We say that $u$ is a {\it local parallel refractor} if for each $x_0\in \Omega$ there is $Y\in \Sigma$ and $\delta>0$,
%such that $u(x)\leq \phi(x,Y,X_0)$ for $|x-x_0|<\delta$, with $X_0=(x_0,u(x_0))$.
%Obviously, a parallel refractor is a local refractor. But the converse is in general not true, an example is given in
For example, rotating the Figure 
\ref{fig:localnonglobal} around the $z$-axis,
\begin{figure}[htp]
\begin{center}
   \includegraphics[width=6in]{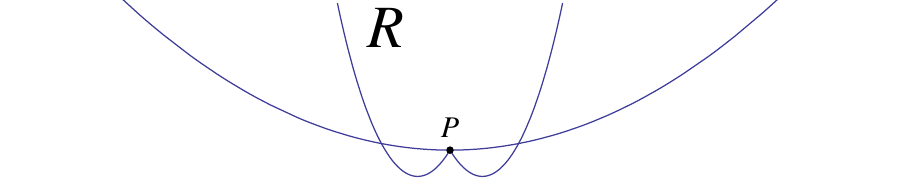}
    %\subfigure[$|X|+2/3 |X-P|=1.7$, $P=(2,0)$]{\label{fig:edge-b}\includegraphics[width=2.9in]{snelllaw.pdf}} 
%    \subfigure[After Sobel edge detection]{\label{fig:edge-c}\includegraphics[scale=1]{SeveralCartesianOvalskappa<1.pdf}}
\end{center}
  \caption{The refractor ${\it R}$ is composed of the minimum of the ellipsoids given by $\phi(x,Y_1,P)$, and $\phi(x,Y_2,P)$, with $Y_1=(.03,5),Y_2=(-.03,5)$ and $P=(0,4.70456)$. The ellipsoid given by $\phi(x,Y_3,P)$ with $Y_3=(0,10)$, supports $\it R$ at $P$ locally but not globally; $\kappa=2/3$.}
  \label{fig:localnonglobal}
\end{figure}
%DRAW PICTURE OF EXAMPLE OF THREE POINTS WHERE LOCAL DOES NOT IMPLY GLOBAL.
%THIS EXAMPLE IS TRUE IN DIMENSION ONE. 
%IN HIGHER DIMENSIONS (3.2) SHOULD IMPLY A SIMILAR INEQUALITY GLOBALLY ????
%By rotation of this example around the $z$-axis, 
we obtain a refractor in 3d that has a local supporting ellipsoid that is not global.
In this case, the target is composed of the circle $x^2+y^2=(.03)^2$ with $z=0$, and the point $(0,0,10)$.

The purpose of this section is to see that under condition \eqref{eq:AWPLnew} below, a local supporting ellipsoid is also global. 
This will be used later in the proof of Theorem \ref{thm:mainestimate}.
%what conditions a local refractor is global.
%\section{Appendix}\label{sec:appendix}

The target $\Sigma$ satisfies the condition AW from $X_0\in C_{\Omega}$ if for all $Y_0\in\Sigma$, and for all vectors $\xi$ and $\eta$ such that $\xi\perp \eta$ we have 
\begin{equation}\label{eq:AWPLnew}
\dfrac{d^{2}}{d\epsilon^{2}}\left. \left\langle D_x^{2}\phi(x_0,Y_\epsilon,X_0)\eta,\eta\right\rangle \right|_{\epsilon=0}\leq 0,
\end{equation}
where $Y_{\epsilon}=X_0+s(\Lambda(v+\epsilon\xi))\Lambda(v+\epsilon\xi)$ and $v=D\phi(x_0,Y_0,X_0)$.
Clearly \eqref{eq:AWPL} implies  \eqref{eq:AWPLnew}.
We will show in Proposition \ref{prop:global} that \eqref{eq:AWPLnew} implies that for all $\bar Y,\hat Y\in \Sigma$
we have
\begin{equation}\label{eq:convexityconditiononthetargetbisbisbis}
\phi(x,Y_{X_0}(\lambda),X_0)\geq\min\left\{\phi(x,\bar Y,X_0),\phi(x,\hat Y,X_0)\right\}\end{equation}
for all $x\in \Omega$, and $0\leq \lambda \leq 1$, with $Y_{X_0}(\lambda)=X_0+s_{Z}(\Lambda(v(\lambda)))\,\Lambda(v(\lambda))$. 
%SIMILAR TO A STRONG, CHECK MCCANN-FIGALLI-RECENT PAPER.
%Where $C$ stands for possibly different structural constants.
%THE CONSTANTS C DEPEND ON THE POINTS OR ARE UNIFORM???
Here $\bar v=D_x\phi(x_0,\bar Y,X_0)$, $\hat v=D_x\phi(x_0,\hat Y,X_0)$,
 and $v(\lambda)=(1-\lambda)\bar v +\lambda \hat v$. 
 %Using this we will show that if an ellipsoid supports a refractor locally, then it supports the refractor globally. 
%We shall prove that under the assumption AW, that is, in inequality \eqref{eq:AWPL} the right hand side is replaced by zero (THAT THIS MEANS THAT \eqref{eq:convexityconditiononthetargetbisbis} HOLDS WITHOUT THE QUADRATIC TERM ON THE RHS? MENTION LOEPER?), 

Let 
\begin{equation}\label{eq:definitionofHvX}
H(v,X):=s_X(\Lambda(v))Q(v),
\end{equation} see \eqref{eq:defofLambdav} and Subsection \ref{subsec:assumptiontargetrefractor}. 
%HERE NEED $X$ SUBSCRIPT $s_X$ SEE LATER $H(v,X)$?
%By the condition AW at the points $X$ and $Y(v)=X+s(\Lambda(v))\Lambda(v)$, we mean (THIS NEEDS TO BE PROVED LIKE IN SECTION \ref{sec:example}? OR MOVED THAT HERE?) 
Condition \eqref{eq:AWPLnew} means 
that for all $\eta,\xi\in \R^n$, with $\eta\perp \xi$, we have 
\begin{equation}\label{eq:AWintermsofsecondderivativesxv}
D_{v_l,v_k}(D_{x_i,x_j}\phi(x,Y(v),X)\eta_i\eta_j\xi_k\xi_l\leq 0,
\end{equation}
with $X=X_0=(x_0,x_{n+1}^0)$.
We are going to rewrite this condition in terms of the function $H$.
We will show that this is equivalent to 
\begin{equation}\label{eq:1/Hisconcave}
\langle D_v^{2}(1/H(v,X))\,\xi,\xi\rangle\leq 0,\qquad  \text{for all $\xi\in R^{n}$}.
\end{equation}

Recall that from Subsection \ref{subsec:estimatesderivativesofphi} we have that $D_{x_i}\phi(x,Y,X)=\dfrac{x_i-y_i}{y_{n+1}-x_{n+1}-k|X-Y|}$, where $x_{n+1}=\phi(x,Y,X)$.
%I THINK IS MORE CLEAR HERE TO WRITE:
%Recall that we have $D_i\phi(x,Y,X_0)=\dfrac{x_i-y_i}{y_{n+1}-x_{n+1}-k|X-Y|}$ where $x_{n+1}=\phi(x,Y,X_0)$,
%and $X=(x,\phi(x,Y,X_0))$.???
%I AM THINKING NOW IS BETTER TO LEAVE AS IS WITH $X$.
Let us set
\[
J(Y,X,\eta):=\langle D_x^{2}\phi(x,Y,X)\eta,\eta\rangle=\sum_{i,j=1}^{n}D_{x_i,x_j}\phi(x,Y,X)\eta_i\eta_j.
\]
From \eqref{eq:formulasecondderivativesofphi} and \eqref{eq:formulaforsquareroot} we have
$$D_{x_i,x_j}\phi(x,Y,X)=\delta_{ij}\left(y_{n+1}-x_{n+1}-\kappa\,|X-Y|\right)^{-1}+(1-\kappa^{2})\left(y_{n+1}-x_{n+1}-\kappa |X-Y|\right)^{-3}(x_i-y_i)(x_j-y_j).$$ 
%HERE I GET $(1-\kappa^2)^2$ SEE CALCULATIONS IN SUBSECTION \ref{subsec:estimatesderivativesofphi}
%I ALSO THINK THAT WRITING $X$ IS CONFUSING BECAUSE $X$ IS NOT $(x,x_{n+1})$???
%
%CHECK THE CALCULATIONS TO SEE IF THEY AGREE WITH THE PREVIOUS ONE. 
%I CHECKED IT AND THE FORMULA IS SLIGHTLY DIFFERENT.
Therefore, for $|\eta|=1$ we get
\begin{align*}
J(Y,X,\eta)
=\left(y_{n+1}-x_{n+1}-\kappa |X-Y|\right)^{-1}+(1-\kappa^{2})\left(y_{n+1}-x_{n+1}-\kappa |X-Y|\right)^{-3}\,\langle x-y,\eta\rangle^{2}.
\end{align*}
From \eqref{eq:defofLambdav}, $\Lambda(v)=(-Q(v)v,Q(v)+\kappa)$. Then
 $$Y_\epsilon-X=s_X(\Lambda(v))\,\left(-Q(v)v,Q(v)+\kappa \right).$$
Therefore, $J(Y_\epsilon,X,\eta)=\dfrac{1+(1-\kappa^{2})\,\langle v,\eta\rangle^{2}}{s_X(\Lambda (v))\,Q(v)}:=F(v,X,\eta)$.
%Let us set $F(w,X,\eta)=\dfrac{1+(1-\kappa^{2})\,\langle w,\eta\rangle^{2}}{s_X(\Lambda(w))Q(w)}$. 
We will show that $\dfrac{d^{2}}{d\epsilon^{2}}F(v+\epsilon\xi,X,\eta)|_{\epsilon=0}\leq 0$. 
Calculating the second derivative with respect to $\epsilon$, we have that the last inequality is equivalent to 
\begin{equation}\label{eq:AWbisbis}
\dfrac{d^{2}}{d\epsilon^{2}}F(v+\epsilon\xi,X,\eta)|_{\epsilon=0}=\sum_{k,\ell=1}^{n}D_{v_k,v_\ell}F(v,X,\eta)\xi_k\xi_\ell\leq 0
\end{equation} 
for all vectors $\xi\perp \eta$.
%If we assume the target set $\Sigma$ is given by the graph of a $C^{2}$ function $y_{n+1}=\psi(y)$,
%then we are going to find a condition on $\psi$ so that \eqref{eq:AWbis} holds.
%We set $H(v,X)=s_X(\Lambda(v))Q(v)$.
 % where $s(v,X)$ is the function describing the target $\Sigma$ to be determined.
%We have for $Y\in \Sigma$ that $Y=X+s_X(\Lambda(v))(-Q(v)v,Q(v)+\kappa)$.
%%If we assume the target set $\Sigma$ is given by the graph of a $C^{2}$ function $y_{n+1}=\psi(y)$, 
%%then 
%%for $Y\in \Sigma$ using the expresion $Y=X+s(v,X)(-Qv,Qv+\kappa)$, 
%Then we can write $y_{n+1}=x_{n+1}+s_X(\Lambda(v))(Q(v)+k)=\psi(x-s_X(\Lambda(v))\,Q(v)v)$.
%We therefore get that $H$ satisfies the implicit equation
%%$H(v,X)=s(v,X)Q(v)$ we obtain as implicit definition of $H$ by the formula 
%\begin{equation}\label{eq;implicitformulaforHbis}
%x_{n+1}+H(v,X)\left(\frac{Q(v)+\kappa}{Q(v)}\right)=\psi(x-H(v,X)v).
%\end{equation}
%We are going to find a condition on $\psi$ so that \eqref{eq:AW} holds.

If we set $G=1/H$, then $F(v,X,\eta)=G(v,X)\left(1+(1-\kappa^{2})\,\langle v,\eta\rangle^{2}\right)$.
A calculation gives that 
\begin{align*}
\sum_{k,\ell=1}^{n}D_{v_k,v_\ell}F(v,X,\eta)\xi_k\xi_\ell
&=(1+(1-\kappa^{2})\,\langle v,\eta\rangle^{2})\langle D^{2}G(v,X)\xi,\xi\rangle\\
&\quad +4(1-\kappa^{2})\langle v,\eta\rangle\,\langle\xi,\eta\rangle\,\langle DG(v,X),\xi\rangle
+2(1-\kappa^{2})\,\langle\eta,\xi\rangle^{2}\,G(v,X)\\
&=(1+(1-\kappa^{2})\,\langle v,\eta\rangle^{2})\langle D^{2}G(v,X)\xi,\xi\rangle\\
\end{align*}
since $\xi\perp \eta$.

Therefore we have shown that condition \eqref{eq:AWPLnew} is equivalent to \eqref{eq:1/Hisconcave}.

For simplicity in the notation we assume that $X_0=0$ and consider the solid ellipsoids 
\[
E(Y)=\{X\in \R^{n+1}:c(X,Y)\leq c(0,Y)\},
\]
where $c(X,Y)$ is defined by \eqref{eq:definitionofc(X,Y)}.
%$C(X,Y)=|X-Y|+\kappa(x_{n+1}-y_{n+1})$.
From Subsection \ref{subsec:assumptiontargetrefractor}, we recall  that the target $\Sigma$ is given parametrically from the origin by 
\[
Y=s(\Lambda(v))\Lambda(v),
\] 
where $\Lambda(v)=(-Q(v)v,Q(v)+\kappa)$ and $|\Lambda(v)|=1$, with $Q(v)$ given in \eqref{eq:defofLambdav}.
\begin{proposition}\label{prop:global}
Suppose \eqref{eq:AWPLnew} holds from $X_0$ (assumed for simplicity 0).
Let $\bar Y,\hat Y\in\Sigma$ be  given by $\bar Y=s(\Lambda(\bar v))\Lambda(\bar v)$ and $\hat Y=s(\Lambda(\hat v))\Lambda(\hat v)$, and let $v_{\lambda}=(1-\lambda)\bar v+\lambda\hat v$ for some $\lambda\in (0,1)$, and let $Y_{\lambda}=s(\Lambda(v_{\lambda}))\Lambda( v_{\lambda})$. Set $H(v,0)=H(v)$, given in \eqref{eq:definitionofHvX}.
Then
\begin{equation}\label{eq:concavityof1/H}
\frac{1}{H(v_{\lambda})}\geq (1-\lambda)\frac{1}{H(\bar v)}+\lambda\frac{1}{H(\hat v)},
\end{equation}
%WE NEED TO DEFINE $H$
and
\[
E(Y_{\lambda})\subseteq E(\bar Y)\cup E(\hat Y),
\]
and in particular,
$\phi(x,Y_{\lambda},X_0)\geq\min\{\phi(x,\bar Y,X_0),\phi(x,\hat Y,X_0)\}$ for all $x$ in their common domain (in particular for $x\in \Omega$).
%\marginpar{$X_0=0$?}
\end{proposition}

\begin{proof}
Inequality \eqref{eq:concavityof1/H} follows from \eqref{eq:1/Hisconcave} which is equivalent to \eqref{eq:AWPLnew}.

We first notice that the set $\text{\it bdry} E(\hat Y)\cap\text{\it bdry} E(Y_{\lambda})$ is contained on a hyperplane $\hat T$. Indeed, for $X\in \text{\it bdry} E(\hat Y)\cap\text{\it bdry} E(Y_{\lambda})$ we have that  
$c(X,\hat Y)=c(0,\hat Y)$, so $|X-\hat Y|^{2}=(|\hat Y|-\kappa x_{n+1})^{2}$ which gives that 
$|X|^2-2\langle X,\hat Y\rangle+2\kappa x_{n+1}|\hat Y|-\kappa^{2}x_{n+1}^{2}=0$. Also $c(X,Y_{\lambda})=c(0,Y_{\lambda})$, and so $|X|^2-2\langle X, Y_{\lambda}\rangle+2\kappa x_{n+1}|Y_{\lambda}|-\kappa^{2}x_{n+1}^{2}=0$.
Subtracting these identities yields 
$\langle X,\hat\eta\rangle=0$ where $\hat\eta=\hat Y-\kappa|\hat Y|e_{n+1}-(Y_{\lambda}-\kappa|Y_{\lambda}|e_{n+1})$, and so 
$\text{\it bdry} E(\hat Y)\cap\text{\it bdry} E(Y_{\lambda})\subseteq \hat T:=\{X:\langle X,\hat\eta\rangle=0\}$.
In the same way, $\text{\it bdry} E(\bar Y)\cap\text{\it bdry} E(Y_{\lambda})\subseteq \bar T$, where $\bar T=\{X:\langle X,\bar\eta\rangle=0\}$ and 
$\bar \eta=\bar Y-\kappa|\bar Y|e_{n+1}-(Y_{\lambda}-\kappa|Y_{\lambda}|e_{n+1})$.

From \eqref{eq:defofLambdav} and the definition of $H$, we can write 
$\hat\eta=\left(H(v_{\lambda})\,v_{\lambda}- H(\hat v)\,\hat v,H(\hat v)-H(v_{\lambda})\right)$ and 
$\bar\eta=\left(H(v_{\lambda})\,v_{\lambda}- H(\bar v)\,\bar v,H(\bar v)-H(v_{\lambda})\right)$.

The proposition will follow from the following claims:

{\bf Claim 1}: If $\langle X,\hat\eta\rangle\geq 0$ and $X\in E(Y_{\lambda})$, then $X\in E(\hat Y)$.

{\bf Claim 2}: If $\langle X,\bar\eta\rangle\geq 0$ and $X\in E(Y_{\lambda})$, then $X\in E(\bar Y)$.

{\bf Claim 3}: If $\langle X,\hat\eta\rangle< 0$ and $\langle X,\bar\eta\rangle< 0$, then $X\notin E(Y_{\lambda})$.

Only the proof of the third claim uses condition AW, i.e., \eqref{eq:concavityof1/H}.

We prove Claim 1.
Let $\langle X,\hat\eta\rangle\geq 0$ and $X\in E(Y_{\lambda})$. 
Since $X\in E(Y_{\lambda})$, we have $|X|^2-2\langle X, Y_{\lambda}\rangle+2\kappa x_{n+1}|Y_{\lambda}|-\kappa^{2}x_{n+1}^{2}\leq 0$; and since $\langle X,\hat\eta\rangle\geq 0$ we have $\langle X,Y_{\lambda}-\hat Y\rangle-\kappa x_{n+1}(|Y_{\lambda}|-|\hat Y|)\leq 0$.
Therefore,
\begin{align*}
&|X|^2-2\langle X,\hat Y\rangle+2\kappa x_{n+1}|\hat Y|-\kappa^{2}x_{n+1}^{2}\\
&=|X|^2-2\langle X, Y_{\lambda}\rangle+2\kappa x_{n+1}|Y_{\lambda}|-\kappa^{2}x_{n+1}^{2} +2\langle X,Y_{\lambda}-\hat Y\rangle-2\kappa x_{n+1}(|Y_{\lambda}|-|\hat Y|)\leq 0.
\end{align*}
It follows that $|X-\hat Y|^{2}\leq (|\hat Y|-\kappa x_{n+1})^{2}$. In this inequality, writing $X=(x,x_{n+1}),
\hat Y=(\hat y, \hat y_{n+1})$, and completing the squares we obtain, since $0<\kappa<1$, that $|\hat Y|-\kappa x_{n+1}\geq 0$ and hence $|X-\hat Y|\leq |\hat Y|-\kappa x_{n+1}$, which means $X\in E(\hat Y)$.

The proof of Claim 2 is exactly the same.

We now prove Claim 3.

Assume that $\langle X,\hat\eta\rangle< 0$ and $\langle X,\bar\eta\rangle< 0$.
Notice that $E(Y_{\lambda})\setminus{0}\subseteq \{X:\langle X,N_{\lambda}\rangle<0\}$, where $N_{\lambda}=(v_{\lambda},-1)$ and hence, it is enough to show that $\langle X,N_{\lambda}\rangle\geq 0$.

We first assume that $H(v_{\lambda})-H(\bar v)\neq 0$ and $H(v_{\lambda})-H(\hat v)\neq 0$.
We will show that we can write 
\[
N_{\lambda}=\frac{(1-t)}{H(v_{\lambda})-H(\bar v)}\bar\eta+\frac{t}{H(v_{\lambda})-H(\hat v)}\hat\eta,
\] 
with 
$\dfrac{(1-t)}{H(v_{\lambda})-H(\bar v)}\leq 0$ and $\dfrac{t}{H(v_{\lambda})-H(\hat v)}\leq 0$, for some $t$.
The above equality holds if and only if 
\[
\frac{(1-t)\left(H(v_{\lambda})\,v_{\lambda}- H(\bar v)\,\bar v\right) }{H(v_{\lambda})-H(\bar v)}+\frac{t\left(H(v_{\lambda})\,v_{\lambda}- H(\hat v)\,\hat v\right) }{H(v_{\lambda})-H(\hat v)}=v_{\lambda},
\]
which holds if and only if 
\[
(1-t)\left(v_{\lambda}+\frac{H(\bar v)\,(v_{\lambda}-\bar v)}{H(v_{\lambda})-H(\bar v)}\right)+t\left(v_{\lambda}+\frac{H(\hat v)\,(v_{\lambda}-\hat v)}{H(v_{\lambda})-H(\hat v)}\right)=v_{\lambda},
\]
which is true if and only if 
\[
(\hat v-\bar v)\left\{\frac{(1-t)\lambda H(\bar v)}{H(v_{\lambda})-H(\bar v)}-\frac{t(1-\lambda)H(\hat v)}{H(v_{\lambda})-H(\hat v)}\right\}=0.
\]
Therefore we choose $t$ such that 
$\dfrac{(1-t)\lambda H(\bar v)}{H(v_{\lambda})-H(\bar v)}=\dfrac{t(1-\lambda)H(\hat v)}{H(v_{\lambda})-H(\hat v)}$.
Since $Q(v)>0$, we have that $\lambda H(\bar v)>0$ and $(1-\lambda)H(\hat v)>0$. Then it follows that $\dfrac{(1-t)}{H(v_{\lambda})-H(\bar v)}$ and $\dfrac{t}{H(v_{\lambda})-H(\hat v)}$ have both the same sign.
From the last identity containing $t$ we obtain 
\[
\lambda H(\bar v)=\frac{t}{H(v_{\lambda})-H(\hat v)}\left\{H(v_{\lambda})\left((1-\lambda)H(\hat v)+\lambda H(\bar v)\right)-H(\bar v)H(\hat v)\right\}.
\]
From \eqref{eq:concavityof1/H} we get 
\[
H(v_{\lambda})\left((1-\lambda)H(\hat v)+\lambda H(\bar v)\right)-H(\bar v)H(\hat v)\leq 0,
\] 
hence 
$
\dfrac{t}{H(v_{\lambda})-H(\hat v)}\leq 0$, and therefore also $\dfrac{(1-t)}{H(v_{\lambda})-H(\bar v)}\leq 0.
$

Next we assume that $H(v_{\lambda})-H(\hat v)=0$ and $H(v_{\lambda})-H(\bar v)\neq 0$. From \eqref{eq:concavityof1/H}, this implies that $H(v_{\lambda})<H(\bar v)$.
If we write $N_{\lambda}=\dfrac{1}{H(v_{\lambda})-H(\bar v)}\bar\eta+t\,\hat\eta$, then 
$t=\dfrac{\lambda H(\bar v)}{(H(v_{\lambda})-H(\bar v))(1-\lambda)H(\hat v)}$, and so $t<0$.

Finally, if $H(v_{\lambda})=H(\hat v)=H(\bar v)$, then $\bar \eta=\lambda H(\bar v)(\hat v-\bar v,0)$, and $\hat \eta=(1-\lambda) H(\bar v)(\bar v-\hat v,0)$. So in this case, both inequalities $\langle X,\hat\eta\rangle< 0$ and $\langle X,\bar\eta\rangle< 0$ are impossible.

%The cases when $H(v_{\lambda})=H(\bar v)$ or $H(v_{\lambda})=H(\hat v)$ follow by the same argument.

This proves Claim 3 and hence the proof of the proposition is complete.

\end{proof}

\begin{proposition}\label{prop:localimpliesglobal}
Suppose that $u$ is a parallel refractor in $\Omega$, $x_0\in \Omega$, and assume that \eqref{eq:AWPLnew} holds from $X_0=(x_0,u(x_0))$.  
If there exist $Y_0\in \Sigma$ and $\epsilon>0$ such $u(x)\leq \phi(x,Y_0,X_0)$ for all $x\in B_\epsilon(x_0)$, then $u(x)\leq \phi(x,Y_0,X_0)$ for all $x\in \Omega$. 
\end{proposition}
\begin{proof}
We define
\[
\partial u(x_0)=\{v\in \R^n: u(x)\leq u(x_0)+ v\cdot (x-x_0)+o(|x-x_0|); \forall x\in B_\epsilon(x_0)\}.
\]
We prove that $\{Y(X_0,v):v\in \partial u(x_0)\}\subset F_u(x_0)$, where $Y(X_0,v)=X_0+s_{X_0}(\Lambda(v))\Lambda(v)$.
Since from \eqref{eq:frecuentfactused} $v=D_x\phi(x,Y,X)$ with $Y\in \Sigma$ if and only if $Y=Y(X,v)$, the inclusion is equivalent to show that
$\partial u(x_0)\subset \{D_x\phi(x_0,Y,X_0):Y\in F_u(x_0)\}:=B$. It is then enough to show that 
the extremal points of $\partial u(x_0)$ are contained in $B$ and that $B$ is convex.
Let $v_0$ be a extremal point of $\partial u(x_0)$, then there exist $x_n\to x_0$ with $u$ differentiable at $x_n$ and $v_n=Du(x_n)\to v_0$, see 
\cite[Theorem 2.5.1]{clarke:bookooptimizationandnonsmoothanalysis}.
Let $X_n=(x_n,u(x_n))$ and let $Y_n\in F_{u}(x_n)$. Since $u$ is differentiable at $x_n$ it follows that $Y_n=Y(X_n,v_n)$ and  then $Y_n\to Y(X_0,v_0):=Y_0$.
We have that $u(x)\leq \phi(x,Y_n,X_n)$ for all $x\in \Omega$. Letting $n\to \infty$ yields $u(x)\leq \phi(x,Y_0,X_0)$ for all $x\in \Omega$, i.e., $Y_0\in F_u(x_0)$ and $D_x\phi(x_0,Y_0,X_0)=v_0$.

To show that $B$ is convex, let $Y_1,Y_2\in F_u(x_0)$ and let $v_i=D_x\phi(x_0,Y_i,X_0)$, $i=1,2$. Consider $v_\lambda=(1-\lambda)v_1+\lambda v_2$ and $Y_\lambda=Y(X_0,v_\lambda)$.
Since $v_\lambda=D_x\phi(x_0,Y_\lambda,X_0)$, we need to show that $Y_\lambda\in F_u(x_0)$. We have 
\[
u(x)\leq \min\{\phi(x,Y_1,X_0),\phi(x,Y_2,X_0)\}\leq \phi(x,Y_\lambda,X_0)
\]
for all $x\in \Omega$ by Proposition \ref{prop:global}. This completes the proof.
\end{proof}

\setcounter{equation}{0}
\section{Main results}\label{sec:definitionofrefractorandmainresults}

%A function $u:\Omega\rightarrow [0,M]$ is called a refractor with respect to $\Sigma$ if for each $x_0\in \Omega$, there exists $Y\in \Sigma$, %such that $u(x)\leq \phi(x,Y,X_0)$ for all $x\in\Omega$ with $X_0=(x_0,u(x_0))$. In this case we say $Y\in F_u(x_0)$.

\begin{lemma}\label{lm:estimateconvexcombinationrefractor}
If $u$ is a refractor, then for each $\bar x,\hat x\in\Omega$ and each $s\in [0,1]$ we have $u((1-s)\bar x+s\hat x)\geq (1-s)u(\bar x)+su(\hat x)-C|\bar x-\hat x|^{2}s(1-s)$, with $C$ is a structural constant.
\end{lemma}

\begin{proof}

We use the fact from Subsection \ref{subsec:estimatesderivativesofphi} that $|D_{x_i x_j}\phi(x,Y,\bar X)|\leq C$ for any $x\in\Omega$, $\bar X\in C_{\Omega}$ and $Y\in\mathcal T$.

Given $\tilde x\in\Omega$, let $\tilde Y\in F_u(\tilde x)$ and set $\tilde X=(\tilde x,u(\tilde x))$. We have $u(x)\leq \phi(x,\tilde Y,\tilde X)\leq \phi(\tilde x,\tilde Y,\tilde X)+\langle D\phi(\tilde x,\tilde Y,\tilde X),x-\tilde x\rangle+C|x-\tilde x|^{2}=u(\tilde x)+\langle p,x-\tilde x\rangle+C|x-\tilde x|^{2}$ for all $x\in\Omega$.
Next,  given $s\in[0,1]$, let $x_s=(1-s)\bar x+s\hat x$, $Y\in F_u(x_s)$, and $X_s=(x_s,u(x_s))$. Applying the previous inequality $\tilde x \rightsquigarrow x_s$, $\tilde Y \rightsquigarrow Y$,  and $p\rightsquigarrow D\phi(x_s,Y,X_s)$, we get $u(x)\leq u(x_s)+\langle p,x-x_s\rangle+C|x-x_s|^{2}$ for all $x\in\Omega$.

Let $\psi(t)=u(x_t)$. We have $\psi(t)\leq u(x_s)+\langle p,x_t-x_s\rangle+C|x_t-x_s|^{2}=\psi(s)+(t-s)\langle p,\hat x-\bar x\rangle+C(t-s)^{2}|\bar x-\hat x|^{2}$. Therefore $\psi(0)\leq \psi(s)-s\langle p,\hat x-\bar x\rangle+Cs^{2}|\bar x-\hat x|^{2}$ and $\psi(1)\leq \psi(s)+(1-s)\langle p,\hat x-\bar x\rangle+C(1-s)^{2}|\bar x-\hat x|^{2}$. Proceeding as in Theorem \ref{thm:awforperpendicularvectors}, we obtain $(1-s)\psi(0)+s\psi(1)\leq \psi(s)+Cs(1-s)|\bar x-\hat x|^{2}$ and the lemma is proved.
\end{proof}

We will next prove our main lemma.

\begin{lemma}\label{lm:mainlemmarefractor}

Suppose $u$ is a parallel refractor and the target $\Sigma$ is regular from $X^{\star}=(x^{\star},u(x^{\star}))$ in the sense of Definition \ref{def:localtargetbis}. There exist constants $\delta$, $C_1$ and $C_2$ depending on $X^{\star}$, such that $B_\delta(x^\star)\subset \Omega$, and if $\bar x,\hat x\in B_{\delta}(x^{\star})$,  $\bar Y\in F_u(\bar x),\hat Y\in F_u(\hat x)$ with $|\bar Y-\hat Y|\geq |\bar x-\hat x|$,
%I THINK HERE WE NEED TO ASSUME $$|\bar Y-\hat Y|\geq \alpha\, |\bar x-\hat x|$$ FOR SOME $\alpha>0$ BECAUSE OF THE NEXT THEOREM.
then there exists $x_0\in \overline{\bar x\hat x}$ (the segment from $\bar x$ to $\hat x$) such that if $X_0^{\star}=(x_0,u(x_0))$, then  
\begin{equation}\label{eq:mainlemmainequality}
u(x)-\phi(x,Y,X_0^{\star})\leq C\,|\bar Y-\hat Y|\,|\bar x-\hat x|+C\,|Y(\lambda)-Y|\,|x-x_0|-C_1\,|\bar Y-\hat Y|^{2}|x-x_0|^{2},
\end{equation} for all 
$Y(\lambda)\in [\bar Y,\hat Y]_{X_0^{\star}}$, $1/4\leq \lambda \leq 3/4$, $Y\in\Sigma$ and for all $x\in \Omega\cap B_{C_2}(x_0)$.

We remark that $C$ is a structural constant, depending only on the bounds for the derivatives of $\phi$.
%, ALSO DEPENDING ON $\alpha$.
\end{lemma}
\begin{proof}
Since $\Sigma$ is regular from $X^{\star}$, there exists a neighborhood $U_{X^{\star}}$ of $X^{\star}$ such that 
\eqref{eq:convexityconditiononthetargetbisbis} holds for all $\bar Y,\hat Y\in \Sigma$ and all $Z\in U_{X^{\star}}$.
Since parallel refractors are uniformly Lipschitz in $\Omega$, there exists $\delta>0$ such that $(x,u(x))\in U_{X^{\star}}$ for all $x\in B_\delta(x^\star)$.

%Set $\bar X=(\bar x,u(\bar x))$ and $\hat X=(\hat x,u(\hat x))$. 
If $\bar Y\in F_u(\bar x)$ and $\hat Y\in F_u(\hat x)$,  it follows that $u(x)\leq\min\{\phi(x,\bar Y,\bar X),\phi(x,\hat Y,\hat X)\}$ for all $x\in\Omega$ with $\bar X=(\bar x, u(\bar x)), \hat X=(\hat x, u(\hat x))$.
By continuity there exists $x_0\in [\bar x,\hat x]$ such that $\phi(x_0,\bar Y,\bar X)=\phi(x_0,\hat Y,\hat X)$. Indeed, setting $h(x)=\phi(x,\bar Y,\bar X)-\phi(x,\hat Y,\hat X)$, we have that $h(\bar x)=u(\bar x)-\phi(\bar x,\hat Y,\hat X)\leq 0$ and $h(\hat x)=\phi(\hat x,\bar Y,\bar X)-u(\hat x)\geq 0$.

Now set $x_{0_{n+1}}=\phi(x_0,\bar Y,\bar X)=\phi(x_0,\hat Y,\hat X)$, $X_0=(x_0,x_{0_{n+1}})$, and recall we have set $X_0^{\star}=(x_0,u(x_0))\in U_{X^{\star}}$ and $u(x_0)\leq x_{0_{n+1}}$. By definition $\phi(x,\hat Y,\hat X)=\phi(x,\hat Y,X_0)$ and $\phi(x,\bar Y,\bar X)=\phi(x,\bar Y,X_0)$ for all $x\in\Omega$. Hence we can write 
\[
u(x)\leq\min\left\{\phi(x,\bar Y,X_0),\phi(x,\hat Y,X_0)\right\}=\min\left\{\phi(x,\bar Y,X_0^{\star}),\phi(x,\hat Y,X_0^{\star})\right\}+E.
\] 
%where $E$ is defined by the equality.
We now estimate $E$. 
First, notice that by Lemma \eqref{lm:lipschitzinX0}, we have $0\leq E\leq C(x_{0_{n+1}}-u(x_0))$.
We claim that 
\begin{equation}\label{eq:claimmainlemma}
x_{0_{n+1}}-u(x_0)\leq  C\,|\bar Y-\hat Y||\bar x-\hat x|.
\end{equation}
We can write $x_0=(1-s)\bar x+s\hat x$ for some $0<s<1$. Then, by Lemma \eqref{lm:estimateconvexcombinationrefractor}, we have $u(x_0)\geq(1-s)u(\bar x)+su(\hat x)-C|\bar x-\hat x|^{2}s(1-s)$. Since $u(\bar x)=\phi(\bar x,\bar Y,X_0)$ and $u(\hat x)=\phi(\hat x,\hat Y,X_0)$, we get 
\begin{equation}\label{eq:middleestimatemainlemma}
x_{0_{n+1}}-u(x_0)\leq x_{0_{n+1}}-\left( (1-s)\phi(\bar x,\bar Y,X_0)+s\phi(\hat x,\hat Y,X_0)\right)+C|\bar x-\hat x|^{2}s(1-s).
\end{equation}
On the other hand, $\phi(\hat x,\hat Y,X_0))\geq \phi(x_0,\hat Y,X_0))+\langle D\phi(x_0,\hat Y,X_0),\hat x-x_0\rangle$, and 
$\phi(\bar x,\bar Y,X_0))\geq \phi(x_0,\bar Y,X_0))+\langle D\phi(x_0,\hat Y,X_0),\bar x-x_0\rangle$ by convexity.
Using that $\bar x-x_0=s(\bar x-\hat x)$, and $\hat x-x_0=-(1-s)(\bar x-\hat x)$ we obtain 
\begin{align*}
\phi(\hat x,\hat Y,X_0)&\geq x_{0_{n+1}}-(1-s)\langle D\phi(x_0,\hat Y,X_0),\bar x-\hat x\rangle,\text{ and }\\
\phi(\bar x,\bar Y,X_0))&\geq x_{0_{n+1}}+s\langle D\phi(x_0,\bar Y,X_0),\bar x-\hat x\rangle.
\end{align*}
It follows that 
\begin{align*}
&(1-s)\phi(\bar x,\bar Y,X_0))+s\phi(\hat x,\hat Y,X_0)\\
&\geq x_{0_{n+1}}+(1-s)s\,\left\langle \left(D\phi(x_0,\bar Y,X_0)-D\phi(x_0,\hat Y,X_0)\right),\bar x-\hat x\right\rangle\\
&\geq x_{0_{n+1}}-s(1-s)\,\left| \left\langle \left(D\phi(x_0,\bar Y,X_0)-D\phi(x_0,\hat Y,X_0)\right),\bar x-\hat x\right\rangle\right|\\
&\geq x_{0_{n+1}}-s(1-s)\,\left|D\phi(x_0,\bar Y,X_0)-D\phi(x_0,\hat Y,X_0)\right|\,|\bar x-\hat x|\\
&\geq x_{0_{n+1}}-s(1-s)\,C\,|\bar Y-\hat Y|\,|\bar x-\hat x|,
\end{align*}
from the estimates for the derivatives of $\phi$ in $Y$. 
%HERE WE USE THE SEGMENT BETWEEN TWO POINTS IN $\Sigma$ IS CONTAINED IN $\mathcal T$.
Therefore, inserting the last estimate in \eqref{eq:middleestimatemainlemma} we obtain
$x_{0_{n+1}}-u(x_0)\leq Cs(1-s)\left(|\bar Y-\hat Y||\bar x-\hat x|+|\bar x-\hat x|^{2}\right)\leq C|\bar Y-\hat Y||\bar x-\hat x|$, where in the last inequality we have used that $s\in [0,1]$ and $|\bar x-\hat x|\leq |\bar Y-\hat Y|$ by assumption. We then obtain claim \eqref{eq:claimmainlemma}.

This yields
\begin{align*}
u(x)&\leq \min\{\phi(x,\bar Y,X_0^{\star}),\phi(x,\hat Y,X_0^{\star})\}+C|\bar Y-\hat Y||\bar x-\hat x|.
%&\leq (1-\lambda)\phi(x,\bar Y,X_0^{\star})+\lambda\phi(x,\hat Y,X_0^{\star})+C|\bar Y-\hat Y||\bar x-\hat x|,
\end{align*}
%for $0\leq \lambda\leq 1$.

If $\bar x, \hat x\in B_\delta(x^\star)$, then $X_0^*\in U_{X^\star}$. So
if $Y(\lambda)\in [\bar Y,\hat Y]_{X_0^{\star}}$, then we can apply \eqref{eq:convexityconditiononthetargetbisbis} to get  
$\min\{\phi(x,\bar Y,X_0^{\star}),\phi(x,\hat Y,X_0^{\star})\}\leq\phi(x,Y(\lambda),X_0^{\star}) - C_1\,|\bar Y-\hat Y|^{2}|x-x_0|^{2}$, for $x\in B_{C_2}(x_0)\cap\Omega$ where the constants $C_1,C_2$ depend on $X^{\star}$.
 And so 
\[
u(x)\leq \phi(x,Y(\lambda),X_0^{\star})+C|\bar Y-\hat Y||\bar x-\hat x|-C_1\,|\bar Y-\hat Y|^{2}|x-x_0|^{2}.
\]

Finally, from Lemma \eqref{lm:estimateslipinYandx}, we have that $|\phi(x,Y(\lambda),X_0^{\star})-\phi(x,Y,X_0^{\star})|\leq C|Y(\lambda)-Y||x-x_0|$, and consequently
$u(x)-\phi(x,Y,X_0^{\star})\leq C\,|\bar Y-\hat Y||\bar x-\hat x|+C\,|Y(\lambda)-Y||x-x_0|-C_1\,|\bar Y-\hat Y|^{2}|x-x_0|^{2}$, where the constant $C$ is structural.
This completes the proof of the lemma.

\end{proof}

We are now a position to prove our main theorem.
\begin{theorem}\label{thm:mainestimate}
Suppose $u$ is a parallel refractor, and the target $\Sigma$ is regular from $X^{\star}=(x^{\star},u(x^{\star}))$ in the sense of Definition \ref{def:localtargetbis}.
%, and \eqref{eq:curveisaconnectedset} holds. 
Let $C, C_1, C_2$ and $\delta$ be the constants of Lemma \ref{lm:mainlemmarefractor}. There exists a constant $M$ depending on $C$ and $C_1$ such that if $\hat x,\bar x\in B_{\frac{\delta}{2}}(x^{\star})$, $\bar Y\in F_u(\bar x),\hat Y\in F_u(\hat x)$ are such that
\begin{equation}\label{eq:differenceofYbiggerthandifferenceofxwithdistance}
\dfrac{|\bar Y-\hat Y|}{|\bar x-\hat x|}\geq \max\left\{1, \left( \dfrac{2M}{\delta}\right)^2, \left(\dfrac{2M}{C_2}\right)^2 \right\} ,
\end{equation}
then there exists $x_0$ on the straight segment $\overline{\bar x\,\hat x}$
%\marginpar{$[]$ needs\\subindex?\\or is\\straight\\segment?} 
such that we have 
\[
N_{\mu}\left(\left\{Y(\lambda)\in [\bar Y,\hat Y]_{X_0^{\star}}:\lambda\in [1/4,3/4]\right\}\right)\cap\Sigma\subseteq F_u(B_{\eta}(x_0)),
\] 
with $X_0^{\star}=(x_0,u(x_0))$, $\mu=|\bar Y-\hat Y|^{\frac{3}{2}}|\bar x-\hat x|^{\frac{1}{2}}$ and $\eta=M\,\dfrac{|\bar x-\hat x|^{\frac{1}{2}}}{|\bar Y-\hat Y|^{\frac{1}{2}}}$.
\end{theorem}

\begin{proof}
From the assumption, we have $|\bar Y-\hat Y|\geq |\bar x-\hat x|$ and so Lemma \ref{lm:mainlemmarefractor} is applicable.
Let $x_0$ be the point in that lemma.

Fix $Y\in N_{\mu}(\{Y(\lambda)\in [\bar Y,\hat Y]_{X_0^{\star}}:\lambda\in [\frac{1}{4},\frac{3}{4}]\})\cap\Sigma$.
So, there exists $Y(\lambda)\in[\bar Y,\hat Y]_{X_0^{\star}}$ with $\lambda\in [\frac{1}{4},\frac{3}{4}]$ such that $|Y(\lambda)-Y|<\mu$.

We then have from Lemma \ref{lm:mainlemmarefractor} that 
\begin{align*}
u(x)-\phi(x,Y,X_0^{\star})&\leq C\, |\bar Y-\hat Y||\bar x-\hat x|+C\,|Y(\lambda)-Y||x-x_0|-C_1\lambda(1-\lambda)|\bar Y-\hat Y|^{2}|x-x_0|^{2}\\
&\leq C\,|\bar Y-\hat Y||\bar x-\hat x|+C\,\mu|x-x_0|-\frac{C_1}{16}|\bar Y-\hat Y|^{2}|x-x_0|^{2},
\end{align*} 
for all $x\in\Omega$ such that $|x-x_0|<C_2$.

The right hand side in the last inequality is strictly negative for all $x\in \Omega$, with $|x-x_0|\geq \eta$, if we choose $\eta>
\dfrac{8C\,\mu+4\, \sqrt{4\,C^{2}\,\mu^{2}+C\,C_1|\bar Y-\hat Y|^{3}|\bar x-\hat x|}}{C_1|\bar Y-\hat Y|^{2}}$. We pick $\mu=|\bar Y-\hat Y|^{\frac{3}{2}}|\bar x-\hat x|^{\frac{1}{2}}$, and $\eta:=M\,\dfrac{|\bar x-\hat x|^{\frac{1}{2}}}{|\bar Y-\hat Y|^{\frac{1}{2}}}$ with $M:=2\left(\dfrac{8C + 4\,\sqrt{4C^{2}+CC_1}}{C_1}\right)$.
Since $x_0\in B_{\delta/2}(x^\star)$, it follows from \eqref{eq:differenceofYbiggerthandifferenceofxwithdistance} that
$B_\eta(x_0)\subset B_{\delta}(x^{\star})\subset\Omega$ and $\eta\leq \dfrac{C_2}{2}$.
%This means $C_3\dfrac{|\bar x-\hat x|^{\frac{1}{2}}}{|\bar Y-\hat Y|^{\frac{1}{2}}}\leq \dist(\Omega',\partial \Omega)$.
 Therefore $u(x)-\phi(x,Y,X_0^{\star})< 0$ for all $x\in\left( \Omega\cap B_{C_2}(x_0)\right) \setminus B_\eta(x_0)$.

We now show $Y\in F_u(B_{\eta}(x_0))$. Notice that by definition of $X_0^\star$ we have $u(x_0)=\phi(x_0,Y,X_0^{\star})$.

Let $G_u$ denote the graph of $u$ in $B_{C_2}(x_0)$, that is, $G_u=\{(x,u(x)):x\in\Omega\cap B_{C_2}(x_0)\}$.  Consider 
\begin{equation}\label{eq:infimumofdifferenceofc}
\inf\{c(X,Y)-c(X_{0}^{\star},Y):X\in G_u\},
\end{equation} 
recalling that $c(X,Y)=|X-Y|+k(x_{n+1}-y_{n+1})$.
 We claim that $c(X,Y)-c(X_{0}^{\star},Y)\geq 0$ for $X=(x,u(x))$ with $|x-x_0|\geq \eta$. 
%GIVE A PROOF. 
Indeed, let $\bar X=(x,\phi(x,Y,X_0^{\star}))$ and notice that $c(X_0^{\star},Y)=c(\bar X,Y)$. Since $\Sigma\subset \mathcal T$, with $\mathcal T$ given by \eqref{eq:definitionofregionT}, we have $X\in E^{-}(Y,c(X,Y))$.
Since $u(x)-\phi(x,Y,X_0^{\star})\leq 0$ for $|x-x_0|\geq \eta$, we have that $X$ is below $\bar X$, and therefore we must have $c(X,Y)\geq c(\bar X,Y)$, and the claim follows. 
%Another proof:  $c(X,Y)-c(X_{0}^{\star},Y)=c(X,Y)-c(\bar X,Y)=\frac{\partial c}{\partial x_{n+1}}(\tilde X,Y)(u(x)-\phi(x,Y,X_0^{\star})$, for some $\tilde X\in [X,\bar X]$. Since $\tilde X\in E^{-}(Y,c(\tilde X,Y))$ we have $\frac{\partial c}{\partial x_{n+1}}(\tilde X,Y)\leq 0$ and since also $u(x)-\phi(x,Y,X_0^{\star})\leq 0$ we have the claim.
Therefore, the infimum in \eqref{eq:infimumofdifferenceofc} is attained at some point $\tilde X=(\tilde x,u(\tilde x))$ with $\tilde x\in B_{\eta}(x_0)$. 

We show that $Y\in F_u(\tilde x)$. 
%GIVE A PROOF.
Indeed, we have $c(X,Y)\geq c(\tilde X,Y)$ for all $X\in G_u$.
Writing $X=(x,u(x))$ we have that $(|x-y|^{2}+(y_{n+1}-u(x))^{2})^{1/2}-k(y_{n+1}-u(x))\geq c(\tilde X,Y)$ and noticing that $y_{n+1}\geq u(x)$, we get $u(x)\leq y_{n+1}-\dfrac{kc(\tilde X,Y)}{1-k^{2}}-
\sqrt{\dfrac{c(\tilde X,Y)^{2}}{(1-k^{2})^{2}}-\dfrac{|x-y|^{2}}{1-k^{2}}}=\phi(x,Y,\tilde X)$ for $x\in \Omega\cap B_{C_2}(x_0)$.
%\marginpar{with appendix\\proof\\would\\end\\here?}
Since $\eta<C_2/2$, we have that $B_{C_2/2}(\tilde x)\subset B_{C_2}(x_0)$.
We therefore obtain the local estimate $u(x)\leq \phi(x,Y,\tilde X)$ for all $x\in B_\epsilon(\tilde x)$, with $\epsilon$ small.

Since $\tilde x\in B_{\delta}(x^{\star})$, it follows from the choice of $\delta$ in Lemma \ref{lm:mainlemmarefractor}
that $\tilde X=(\tilde x,u(\tilde x))\in U_{X^{\star}}$, the neighborhood in Remark \ref{rmk:localequivalenceperpvectors}. Therefore we can apply \eqref{eq:AWPL} with $Z=\tilde X$, and in particular \eqref{eq:AWPLnew} holds at $\tilde X$. Hence
applying Proposition \ref{prop:localimpliesglobal} with $X_0\rightsquigarrow \tilde X$, we obtain that $u(x)\leq \phi(x,Y,\tilde X)$ holds for all $x\in \Omega$ obtaining that 
$Y\in F_u(\tilde x)$.

%From the assumption \eqref{eq:curveisaconnectedset} made on the refractor this implies that $Y\in F_u(\tilde x)$.
%%REVISE KEEPING IN MIND WHICH ASSUMPTION WE CHOOSE.
%To see this,
%first notice that $u(x)<\phi(x,Y,X_0^\star)\leq \phi(x,Y,\tilde X)$ for all $|x-x_0|=\eta$. Suppose, by contradiction, that $Y\notin F_u(\tilde x)$. Then there exists $\bar x\in \Omega$ with $|\bar x-x_0|>\eta$ such that $u(\bar x)=\phi(\bar x,Y,\tilde X)$. Consider $[\tilde x,\bar x]_{Y}$. By Lemma \ref{lm:contactset} we have that $u(x)\geq \phi(x,Y,\tilde X)$ for all $x\in [\tilde x,\bar x]_{Y}$, but by assumption this is a connected set and hence it must intersect $\partial B_{\eta}(x_0)$.
%This contradicts that $u(x)<\phi(x,Y,\tilde X)$ for all $x\in \partial B_{\eta}(x_0)$.
\end{proof}

\subsection{A property of refractors}

Let $\sigma$ denote  the Borel measure given on the target $\Sigma$ and let $\mu$ be a Borel measure in $\Omega$.
We say $x\in \mathcal T_u(Y)$ if and only if $Y\in F_u(x)$, where $F_u$ is defined by \eqref{eq:definitionofrefractormapping}.
Assuming $\mu=f\,dx$ with $f\in L^1(\Omega)$ and the energy conservation condition 
\[
\mu(\Omega)=\sigma(\Sigma),
\]
it is proved in \cite{gutierrez-tournier:parallelrefractor} the existence of a refractor $u$ such that 
\begin{equation}\label{eq:definitionofparallelrefractor}
\mu(\mathcal T_u(E^{\star}))=\sigma(E^{\star}),\qquad \text{for all Borel subsets $E^{\star}\subset \Sigma$}.
\end{equation}
%We recall the definition of solution to the parallel refractor problem using the tracing mapping \cite{gutierrez-tournier:parallelrefractor}.  
%Let $\mu$ be a Borel measure in $\Omega$.
%We have a set $\Omega$ and a measure $\mu$ and a target set $\Sigma$ and a measure $\sigma$. So as a mapping of sets $F_u:\Omega\rightarrow \Sigma$ and $T_u:\Sigma\rightarrow \Omega$. 
%{\it A parallel refractor} $u$ is a solution to $\mu(\mathcal T_u(E^{\star}))=\sigma(E^{\star})$ for all Borel subsets $E^{\star}$ of $\Sigma$; the existence of solutions is proved in \cite{gutierrez-tournier:parallelrefractor}. 

The purpose of this subsection is to show the following proposition.

\begin{proposition}\label{prop:mu(S)=0}
Suppose $u$ is a refractor solving \eqref{eq:definitionofparallelrefractor}, and define 
\[
S=\{x\in\Omega:\text{ there exists } \bar x\neq x,\bar x\in\Omega,\text{ such that }   F_u(x)\cap F_u(\bar x)\neq\emptyset\}.
\]
Let us assume the following conditions on $\Sigma$ and $\sigma$:
\begin{enumerate}
\item[(a)] $\Sigma$ is the graph of a $C^{1}$ function, say $\Sigma=\{(y,\psi(y)):y\in\Omega^{\star}\}$ with $\Omega^*$ some domain in $\R^n$; 
\item[(b)] Given $Y\in \Sigma$, let $T_Y$ denote the tangent plane to $\Sigma$ at $Y$. Assume that for each $Y\in\Sigma$ and for each $X\in C_{\Omega}$, the line $\{X+s(Y-X),s\in R\}$ is not contained in $T_Y$, that is, this line intersects $T_Y$ only at the point $Y$. 
\footnote{This condition is implied by the visibility condition in Lemma  \ref{lm:estimatesofpointsintarget} because $Y\in \Sigma$ and some $X_0\in C_\Omega$ the line joining $Y$ and $X_0$ is contained in $T_Y$, then there is ball $B$ centered at $X_0$ with $B\subset C_\Omega$. By the visibility condition the convex hull $\mathcal C$ of $Y$ and $B$ intersects $\Sigma$ only at $Y$. But then the line joining $Y$ and $X_0$ is contained in $\mathcal C$ and $T_Y$. Therefore $\Sigma$ is not differentiable at $Y$.};
\item[(c)] if $E\subset \Sigma$ with $|\{y\in \Omega^*: (y,\psi(y))\in E\}|=0$, then $\sigma(E)=0$.
\end{enumerate}
Then $\mu(S)=0$ and we have the inequality
\begin{equation}\label{eq:inequalityforsigmaandmu}
\sigma(F_u(B))\leq \mu(B),\qquad \text{for all balls $B\subset \Omega$.}
\end{equation}
\end{proposition}
\begin{proof}
We first notice that $F_u(B)$ is a Borel set for each closed ball $B$. 
Because if $K\subseteq \Omega$ is closed, then $F_u(K)$ is closed in $\Sigma$. In fact, let $Y_k\in F_u(x_k)$ with $x_k\in K$ and assume $Y_k\rightarrow Y$ with $Y\in \Sigma$. There exists a subsequence $x_{k_j}\rightarrow \bar x$ for some $\bar x\in K$. Setting $X_{k_j}=(x_{k_j},u(x_{k_j}))$, we have $X_{k_j}\rightarrow \bar X$ where $X=(\bar x,u(\bar x))$.
Since $u(x)\leq \phi(x,Y_k,X_k)$ for all $x\in\Omega$ and for all $k$, we have $u(x)\leq \phi(x,Y,\bar X)$. Therefore $Y\in F_u(K)$.

Let us assume for a moment that $\mu(S)=0$. 
It is easy to see that that $\mathcal T_u(F_u(B))\subseteq B\cup S$,
and therefore, \eqref{eq:inequalityforsigmaandmu} follows.

To prove that $\mu(S)=0$, let us consider the set
%Let $S=\{x\in\Omega:\exists \bar x\neq x,\bar x\in\Omega,\   F_u(x)\cap F_u(\bar x)\neq\emptyset\})$ and 
\[
S^{\star}=\{Y\in\Sigma:\exists x,\bar x \in\Omega, x\neq\bar x\  \ Y\in F_u(x)\cap F_u(\bar x)\}.
\]  
We have that $\mathcal T_u(S^{\star})=S$ and therefore $\mu(S)=\sigma(S^{\star})$.
Under 
the assumptions (a), (b) and (c) above, we are going to show that $\sigma(S^{\star})=0$.
%certain conditions on the measure $\sigma$ and the target $\Sigma$.
%Indeed, we assume the following on the target $\Sigma$:
%\begin{enumerate}
%\item[(a)] $\Sigma$ is the graph of a $C^{1}$ function, say $\Sigma=\{(y,\psi(y)):y\in\Omega^{\star}\}$ with $\Omega^*$ is some domain in $\R^n$; and
%\item[(b)] Given $Y\in \Sigma$, let $T_Y$ denote the tangent plane to $\Sigma$ at $Y$. Assume that for each $Y\in\Sigma$ and for each $X\in C_{\Omega}$, the line $\{X+s(Y-X),s\in R\}$ is not contained in $T_Y$, that is, this line intersects $T_Y$ only at the point $Y$. THIS CONDITION IS IMPLIED BY THE VISIBILITY CONDITION IN LEMMA  \ref{lm:estimatesofpointsintarget} BECAUSE IF FOR SOME $Y\in \Sigma$ AND SOME $X_0\in C_\Omega$ the line joining $Y$ and $X_0$ is contained in $T_Y$, then there is ball $B$ centered at $X_0$ with $B\subset C_\Omega$. By the visibility condition the convex hull $\mathcal C$ of $Y$ and $B$ intersects $\Sigma$ only at $Y$. But then the line joining $Y$ and $X_0$ is contained in $\mathcal C$ and $T_Y$. Therefore $\Sigma$ is not differentiable at $Y$.
%\end{enumerate}
%For each $Y\in \Sigma$, let $T_Y$ denote the tangent plane to $\Sigma$ at $Y$. Assume that for each $Y\in\Sigma$ and for each $X\in C_{\Omega}$, the line $\{X+s(Y-X),s\in R\}$ is not contained in $T_Y$, that is, this line intersects $T_Y$ only at the point $Y$. HOW IS THIS RELATED TO VISIBILITY?
%We will also assume $\Sigma$ is the graph of a $C^{1}$ function. Say $\Sigma=\{(y,\psi(y)):y\in\Omega^{\star}\}$ with $\Omega^*$ is some domain in $\R^n$. 
%REWRITE TO SEE WHERE ARE WE GOING.
To this end, we define $u^{\star}:\Omega^{\star}\rightarrow R$ by 
\[
u^{\star}(y)=\inf\{c(X,Y):X\in G_u,Y=(y,\psi(y))\},
\]
where $G_u$ is the graph of $u$.
We claim $u^{\star}$ is Lipschitz in $\Omega^{\star}$.
Indeed, say $u^{\star}(y_0)=c(X_0,Y_0)$ and $u^{\star}(y_1)=c(X_1,Y_1)$, then $u^{\star}(y_1)-u^{\star}(y_0)\leq c(X_0,Y_1)-c(X_0,Y_0)\leq C|Y_1-Y_0|\leq C |y_1-y_0|$, since $\psi$ is Lipschitz. Similarly, we get the other inequality.

Now fix $Y_0\in S^{\star}$ , so $Y_0\in F_u(x)\cap F_u(\bar x)$ with $x\neq \bar x$. Write $Y_0=(y_0,\psi(y_0))$ with $y_0\in\Omega^{\star}$, $X=(x,u(x))$ and $\bar X=(\bar x,u(\bar x))$. We claim that $u^{\star}$ is not differentiable at $y_0$.
Suppose by contradiction that $u^{\star}$ is differentiable at $y_0$.
It is easy to see that $u^{\star}(y_0)=c(X,Y_0)$ and also $u^{\star}(y_0)=c(\bar X,Y_0)$.
From $u^{\star}(y_0)=c(X,Y_0)$, it follows that $u^{\star}(y)\leq c(X,Y)$ for all $y\in\Omega^{\star}$ with equality at $y_0$. Hence we have 
\[
Du^{\star}(y_0)=D_{y}(c(X,(y,\psi(y))))(y_0)=\dfrac{y_0-x+(\psi(y_0)-x_{n+1})D\psi(y_0)}{|Y_0-X|}-kD\psi(y_0).
\]
Also from $u^{\star}(y_0)=c(\bar X,Y_0)$, we deduce that 
\[
Du(y_0)=D_{y}(c(\bar X,(y,\psi(y))))(y_0)=\dfrac{y_0-\bar x+(\psi(y_0)-\bar x_{n+1})D\psi(y_0)}{|Y_0-\bar X|}-kD\psi(y_0).
\] 
This implies that $\dfrac{y_0-x+(\psi(y_0)-x_{n+1})D\psi(y_0)}{|Y_0-X|}=\dfrac{y_0-\bar x+(\psi(y_0)-\bar x_{n+1})D\psi(y_0)}{|Y_0-\bar X|}$.
Let us set $\Gamma=\dfrac{X-Y_0}{|X-Y_0|}=(\xi,\xi_{n+1})$ and $\bar \Gamma=\dfrac{\bar X-Y_0}{|\bar X-Y_0|}=(\bar \xi,\bar \xi_{n+1})$. 
So $\bar \xi+\bar \xi_{n+1}D\psi(y_0)=\xi+\xi_{n+1}D\psi(y_0)$ and hence $\bar \xi-\xi=D\psi(y_0)(\xi_{n+1}-\bar \xi_{n+1})$.

If $\xi_{n+1}=\bar \xi_{n+1}$, then $\xi=\bar\xi$. Hence $\Gamma=\bar \Gamma$ but this implies that $X=\bar X$, since $X,\bar X\in E^{-}(Y_0,b)$ where $b=c(X,Y_0)=c(\bar X,Y_0)$.  Since by assumption $X\neq\bar X$, we obtain a contradiction.

Therefore, $\xi_{n+1}\neq\bar \xi_{n+1}$ and so we can write $D\psi(y_0)=\dfrac{\bar\xi-\xi}{\xi_{n+1}-\bar \xi_{n+1}}$.
We claim that the line $L:=\left\{Y_0+s\left(\frac{\Gamma}{2}+\frac{\bar\Gamma}{2}\right):s\in R\right\}$ is contained in $T_{Y_0}$.
Indeed, if $Y=Y_0+s\left(\frac{\Gamma}{2}+\frac{\bar\Gamma}{2}\right)$, then a simple calculation shows that 
$\left\langle Y-Y_0,(-D\psi(y_0),1)\right\rangle=0$. 
%Then $L\subset T_{Y_0}$.
On the other hand, the line $L$ clearly intersects $C_{\Omega}$. We then obtain a contradiction with the assumption (b) above, and
% and this contradicts the hypothesis. 
therefore $u^{\star}$ is not differentiable at $y_0$.

If we set $P^{\star}=\{y\in\Omega^{\star}:(y,\psi(y))\in S^{\star}\}$, then we proved that $y\in P^{\star}$ implies that $u^{\star}$ is not differentiable at $y$.
Since $u^{\star}$ is Lipschitz in $\Omega^{\star}$, we get $|P^{\star}|=0$.
Therefore from (c) we obtain $\sigma(S^{\star})=0$ which completes the proof of the proposition.

%We now introduce the following assumption on the measure $\sigma$:
%\begin{enumerate}
%\item[(c)] if $E\subset \Sigma$ with $|\{y\in \Omega^*: (y,\psi(y))\in E\}|=0$, then $\sigma(E)=0$.
%\end{enumerate}
%If this holds we then get $\sigma(S^{\star})=0$.
%
%REMOVE THIS. If for instance, we assume that for each $E^{\star}\subseteq\Sigma$ we have $\sigma(E^{\star})\leq C |\{y\in\Omega^{\star}:(y,\psi(y))\in E^{\star}\}$, then condition (c) holds.
%
%Consequently, under (a), (b), and (c) we get that $\mu(S)=0$, since $\mu(S)=\sigma(S^{\star})$.
%
%This completes the proof of the claim at the beginning of this section.

% we get $\mu(S)=0$ and we get that for any closed ball $B$ in $\Omega$, that $\sigma(F_u(B))\leq \mu(B)$ and hence, if we also assume that $\mu(B)\leq C |B|$ we then obtain our desired inequality.

\end{proof}

\setcounter{equation}{0}
\section{H\"older continuity of the gradient of the refractor}\label{sec:holderregularityofgradient}
We introduce the following local condition at $X_0\in C_\Omega$ between the measure $\sigma$ and target $\Sigma$:
%We assume the following condition:
There exist a neighborhood $U_{X_0}$ and a constant $\hat C>0$ depending on $X_0$ such that
\begin{equation}\label{eq:conditiononsigmameasureofthetube}
\sigma\left(N_{\mu}\left(\left\{[\bar Y,\hat Y]_{Z}:\lambda\in [1/4,3/4]\right\}\right)\cap\Sigma \right)\geq \hat C\, \mu^{n-1}\,|\bar Y-\hat Y|
\end{equation} 
for any $\bar Y,\hat Y\in\Sigma$, $Z\in U_{X_0}$ and for all $\mu>0$ small (depending on $X_0$). Here $N_{\mu}(E)$ denotes the $\mu$- neighborhood of the set $E$ in $R^{n+1}$.

%If the measure $\mu$ has a density $f\in L^p(\Omega)$,  
%then by H\"older's inequality $\mu(B)\leq \|f\|_p\, |B|^{1/q}$, $1/p +1/q=1$. 
%Assuming the hypotheses of Proposition \ref{prop:mu(S)=0}, we then have from \eqref{eq:inequalityforsigmaandmu} that 
%We introduce the following condition: there exist constants $C>0$ and $q>1$ such that
%\begin{equation}\label{eq:assumptiononmeasuresigma}
%\sigma\left(F_u(B_{\eta})\right)\leq C\, \eta^{n/q}
%\end{equation}
%for all balls $B_{\eta}\subseteq\Omega$. 
%FROM THE PREVIOUS SECTION $\sigma\left(F_u(B_{\eta})\right)\leq C\, \mu(B_\eta)$,
%if the density of $\mu$ is in $L^p$, then the following theorem holds with a exponent depending on $p$.
%THIS IS UNDER CONDITIONS (a), (b) and (c) above.
%CLARIFY THAT WE ONLY NEED A REFRACTOR NOT NECESSARILY SOLUTION TO THE PROBLEM WITH THE MEASURES.
\begin{theorem}\label{thm:holdercontinuity}
 Suppose $u$ is a parallel refractor, the target $\Sigma$ is regular from $X^{\star}=(x^{\star},u(x^{\star}))$ in the sense of Definition \ref{def:localtargetbis}, 
 %\eqref{eq:curveisaconnectedset} holds, 
 and
 there exist constants $C_0>0$ and $1\leq q<\dfrac{n}{n-1}$ such that
\begin{equation}\label{eq:assumptiononmeasuresigma}
\sigma\left(F_u(B_{\eta})\right)\leq C_0\, \eta^{n/q}
\end{equation}
for all balls $B_{\eta}\subseteq\Omega$. Suppose in addition that the local condition \eqref{eq:conditiononsigmameasureofthetube} is satisfied at $X^\star$. 

Then there exist 
constants  $\delta$, $M>0$ and $C_2>0$ depending on $X^{\star}$, such that if $\hat x,\bar x\in B_{\frac{\delta}{2}}(x^{\star})$, $\bar Y\in F_u(\bar x),\hat Y\in F_u(\hat x)$ are such that
\begin{equation}\label{eq:differenceofYbiggerthandifferenceofxwithdistancebis}
\dfrac{|\bar Y-\hat Y|}{|\bar x-\hat x|}\geq \max\left\{1, \left( \dfrac{2M}{\delta}\right)^2, \left(\dfrac{2M}{C_2}\right)^2 \right\} ,
\end{equation}
%and the target $\Sigma$ verifies \eqref{eq:convexityconditiononthetarget}, 
%$\sigma$ satisfies \eqref{eq:conditiononsigmameasureofthetube}, and 
%.  
%and the measure $\mu$ has density $f\in L^p(\Omega)$ for some $n<p\leq \infty$.
%Let $u$ be a refractor.Let $\bar x,\hat x\in\Omega$ and $\bar Y\in F_u(\bar x),\hat Y\in F_u(\hat x)$ and assume that $\frac{\epsilon^{2}}{C_4^{2}}|\bar Y-\hat Y|\geq |\bar x-\hat x|$. 
then we have $|\bar Y-\hat Y|\leq C_1\,|\bar x-\hat x|^\alpha$ with $\alpha=\dfrac{\dfrac{n}{2q}-\dfrac{n-1}{2}}{1+\dfrac32 (n-1)+\dfrac{n}{2q}}$, where $C_1$ depends only on $C_0$ and $\hat C$ in \eqref{eq:conditiononsigmameasureofthetube}, and therefore from $X^\star$.
\end{theorem}
\begin{proof}

From Theorem \ref{thm:mainestimate} we have $N_{\mu}(\{Y(\lambda)\in [\bar Y,\hat Y]_{X_0^{\star}}:\lambda\in [\frac{1}{4},\frac{3}{4}]\})\cap\Sigma\subseteq F_u(B_{\eta}(x_0))$ where $\mu=|\bar Y-\hat Y|^{\frac{3}{2}}|\bar x-\hat x|^{\frac{1}{2}}$ and $\eta=M\dfrac{|\bar x-\hat x|^{\frac{1}{2}}}{|\bar Y-\hat Y|^{\frac{1}{2}}}$. Therefore from \eqref{eq:conditiononsigmameasureofthetube} and \eqref{eq:assumptiononmeasuresigma}, we obtain  $|\bar Y-\hat Y|^{1+\frac32 (n-1)+\frac{n}{2q}}\leq C\,|\bar x-\hat x|^{\frac{n}{2q}-\frac{n-1}{2}}$ and the theorem follows.

\end{proof}
Under all previous hypotheses on the target, we now show interior $C^{1,\alpha}$ estimates.

%We set $\alpha=\frac{1}{4n-1}$

\begin{theorem}\label{thm:c1alpha}
With the assumptions of Theorem \ref{thm:holdercontinuity}, 
%we have that if $u$ is a parallel refractor and  the target $\Sigma$ is regular at $ X^{\star}=(x^{\star},u(x^{\star}))$ in the sense of Definition \ref{def:localtargetbis}.Then, 
there exist positive constants $\delta$ and $C$, depending on $X^{\star}$, such that $u\in C^{1,\alpha}(B_{\delta}(x^{\star}))$ 
with $|Du(\bar x)-Du(\hat x)|\leq C|\bar x-\hat x|^{\alpha}$ for all $\hat x,\bar x\in B_{\delta}(x^{\star})$.
\end{theorem}
\begin{proof}

By Theorem \ref{thm:holdercontinuity} we get that $F_u(x)$ is a singleton for each $x\in B_\delta(x^\star)$. 
%Notice that we can apply  Theorem \ref{thm:holdercontinuity} at $X^{\star}$ and then 
Take $\bar x\in B_{\delta/2}(x^{\star})$.

We first show that if $\bar Y\in F_u(\bar x)$, then $u$ is differentiable at $\bar x$, and $D_{i}u(\bar x)=D_{i}\phi(\bar x,\bar Y,\bar X)$ where $\bar X=(\bar x,u(\bar x))$.

For $0<h<\delta/2$  we have 
\begin{align*}
\frac{u(\bar x+he_i)-u(\bar x)}{h}-D_{i}\phi(\bar x,\bar Y,\bar X)&\leq \frac{\phi(\bar x+he_i,\bar Y,\bar X)-\phi(\bar x,\bar Y,\bar X)}{h}-D_{i}\phi(\bar x,\bar Y,\bar X)\\
&=D_{i}\phi(\bar x+\tilde he_i,\bar Y,\bar X)-D_{i}\phi(\bar x,\bar Y,\bar X), \quad
\text{for some $0\leq\tilde h\leq h$}\\
&=D_{i,i}\phi(\bar x+\hat he_i,\bar Y,\bar X)\tilde h,
\end{align*} 
for some $0\leq \hat h\leq \tilde h$. Hence $\dfrac{u(\bar x+he_i)-u(\bar x)}{h}-D_{i}\phi(\bar x,\bar Y,\bar X)\leq C h$.

To prove the inequality in the opposite direction, let $\hat x=\bar x+he_i$, $\hat Y\in F_u(\hat x)$, and $\hat X=(\hat x,u(\hat x))$. We have that 
\begin{align*}
D_{i}\phi(\bar x,\bar Y,\bar X)-\frac{(u(\hat x)-u(\bar x))}{h}&\leq D_{i}\phi(\bar x,\bar Y,\bar X)-\frac{(\phi(\hat x,\hat Y,\hat X)-\phi(\bar x,\hat Y,\hat X))}{h}\\
&=D_{i}\phi(\bar x,\bar Y,\bar X)-D_{i}\phi(\tilde x,\hat Y,\hat X)
\end{align*}
for some $\tilde x\in [\bar x,\hat x]$. 
From the estimates for the derivatives of $\phi$ from Subsection \eqref{subsec:estimatesderivativesofphi},
we can also write 
\begin{align*}
&D_{i}\phi(\bar x,\bar Y,\bar X)-D_{i}\phi(\tilde x,\hat Y,\hat X)\\
&=D_{i}\phi(\bar x,\bar Y,\bar X)-D_{i}\phi(\tilde x,\bar Y,\bar X)+D_{i}\phi(\tilde x,\bar Y,\bar X)-D_{i}\phi(\tilde x,\hat Y,\bar X)+D_{i}\phi(\tilde x,\hat Y,\bar X)-D_{i}\phi(\tilde x,\hat Y,\hat X)\\
&\leq C(|\bar x-\hat x|+|\bar Y-\hat Y|+|\bar X-\hat X|),
\end{align*} 
%CHECK MIXED DERIVATIVES WITH RESPECT TO $x$ and $X_0$ where 
with $C$ a structural constant.
On the other hand, since $u$ is Lipschitz (with a constant depending only on structure), we have $|\bar X-\hat X|\leq |\bar x-\hat x|+|u(\bar x)-u(\hat x)|\leq C|\bar x-\hat x|$.
In addition, using  Theorem \ref{thm:holdercontinuity}, we get that 
$|\bar Y-\hat Y|\leq C\,\max\{|\bar x-\hat x|,|\bar x-\hat x|^{\alpha}\}\leq C\,h^\alpha$ where 
$C$ is a structural constant depending also on $X^{\star}$. 
%WHY $C$ DEPENDS ON $\epsilon$?
%CHOOSE AT THE BEGINNING $x\in \Omega'\Subset \Omega$ with $\Omega'$ convex so that the constants depend on 
%$\dist (\Omega',\partial \Omega)$, REWRITE.
This yields $\dfrac{u(\bar x+he_i)-u(\bar x)}{h}-D_{i}\phi(\bar x,\bar Y,\bar X)\geq -Ch^{\alpha}$, 
completing the proof that $D_{i}u(\bar x)=D_{i}\phi(\bar x,\bar Y,\bar X)$.

We finally prove that $u\in C^{1,\alpha}\left(B_{\delta/2}(x^{\star})\right)$.
Let $\bar x,\hat x \in B_{\frac{\delta}{2}}(x^{\star})$,
$\bar Y\in F_u(\bar x)$, $\hat Y\in F_u(\hat x)$, and set $\bar X=(\bar x,u(\bar x))$, $\hat X=(\hat x,u(\hat x))$.
%TAKE $\bar x,\hat x\in \Omega'$.
We then have $|D_{i}u(\bar x)-D_{i}u(\hat x)|=|D_{i}\phi(\bar x,\bar Y,\bar X)-D_{i}\phi(\hat x,\hat Y,\hat X)|\leq |D_{i}\phi(\bar x,\bar Y,\bar X)-D_{i}\phi(\hat x,\bar Y,\bar X)|+|D_{i}\phi(\hat x,\bar Y,\bar X)-D_{i}\phi(\hat x,\hat Y,\bar X)|+|D_{i}\phi(\hat x,\hat Y,\bar X)-D_{i}\phi(\hat x,\hat Y,\hat X)|\leq C\{|\bar x-\hat x|+|\bar Y-\hat Y|+
|\bar X-\hat X|\}\leq C|\bar x-\hat x|^{\alpha}$, with a structural constant $C$ depending also on
$X^{\star}$.
%THE CONSTANT DEPENDS ON $\dist(\Omega',\partial \Omega)$.
%
%WHY $C$ DEPENDS ON $\epsilon$?
%
%
%
%We finally prove that $u\in C^{1,\alpha}(\Omega')$.
%Let $\bar x,\hat x \in \Omega'$, $B_{\epsilon}(\bar x)\subseteq \Omega$ and let $\hat x\in B_{\epsilon}(\bar x)$. Let $\bar Y\in F_u(\bar x)$ and $\hat Y\in F_u(\hat x)$ and set $\bar X=(\bar x,u(\bar x))$ and $\hat X=(\hat x,u(\hat x))$
%TAKE $\bar x,\hat x\in \Omega'$.
%
%We then have $|D_{i}u(\bar x)-D_{i}u(\hat x)|=|D_{i}\phi(\bar x,\bar Y,\bar X)-D_{i}\phi(\hat x,\hat Y,\hat X)|\leq |D_{i}\phi(\bar x,\bar Y,\bar X)-D_{i}\phi(\hat x,\bar Y,\bar X)|+|D_{i}\phi(\hat x,\bar Y,\bar X)-D_{i}\phi(\hat x,\hat Y,\bar X)|+|D_{i}\phi(\hat x,\hat Y,\bar X)-D_{i}\phi(\tilde x,\hat Y,\hat X)|\leq C\{|\bar x-\hat x|+|\bar Y-\hat Y|+|\bar X-\hat X|\}\leq C|\bar x-\hat x|^{\alpha}$ where $C$ depends on $\epsilon$.
%THE CONSTANT DEPENDS ON $\dist(\Omega',\partial \Omega)$.
%
%WHY $C$ DEPENDS ON $\epsilon$?

%This finishes the proof.

\end{proof}

\begin{corollary}\label{corollary:applicationtorefractors}
Suppose $u$ is a refractor solving \eqref{eq:definitionofparallelrefractor}, conditions (a), (b), and (c) from Proposition 
\ref{prop:mu(S)=0} hold, and $\mu=f\,dx$ with $f\in L^p(\Omega)$ for some $p>n$.
Suppose that the target $\Sigma$ is regular from $X^{\star}=(x^{\star},u(x^{\star}))$ in the sense of Definition \ref{def:localtargetbis}, 
%\eqref{eq:curveisaconnectedset} holds, 
and the local condition \eqref{eq:conditiononsigmameasureofthetube} is satisfied at $X^\star$.
Then there exist $\delta$ and $C$ depending on $X^{\star}$ such that $u\in C^{1,\alpha}(B_{\delta}(x^{\star}))$ 
%the target $\Sigma$ verifies \eqref{eq:convexityconditiononthetarget}, and $\sigma$ satisfies \eqref{eq:conditiononsigmameasureofthetube}, then $u\in C^{1,\alpha}(\Omega')$ with $\Omega'\Subset \Omega$ and
with $\alpha$ given in Theorem \ref{thm:holdercontinuity}, and $1/p+1/q=1$.
\end{corollary}
\begin{proof}
By H\"older's inequality $\mu(B)\leq \|f\|_p\, |B|^{1/q}$, for all balls $B\subset \Omega$.
Therefore from \eqref{eq:inequalityforsigmaandmu} we obtain condition \eqref{eq:assumptiononmeasuresigma},
and the corollary follows.
\end{proof}

\setcounter{equation}{0}
\section{Example of a target for refraction}\label{sec:example}
%\textbf{A calculation}
%REVISE 

If we assume the target set $\Sigma$ is given by the graph of a $C^{2}$ function $y_{n+1}=\psi(y)$,
then we are going to find a condition on $\psi$ so that \eqref{eq:AWPL} holds locally. 
%We show an example of a target $\Sigma$ satisfying \eqref{eq:AWPL}, and therefore from Remark \ref{rmk:localequivalenceperpvectors} such a target satisfies condition \eqref{eq:convexityconditiononthetargetbisbis}.
We will see that for \eqref{eq:AWPL} to hold, the graph of the target needs to satisfy a quantitative condition, see \eqref{eq:quantitativeconditiononthetarget}. 
As in Subsection \ref{subsec:localandglobalrefractors}, we set $H(v,X)=s_X(\Lambda(v))Q(v)$.
 % where $s(v,X)$ is the function describing the target $\Sigma$ to be determined.
We have for $Y\in \Sigma$ that $Y=X+s_X(\Lambda(v))(-Q(v)v,Q(v)+\kappa)$.
%If we assume the target set $\Sigma$ is given by the graph of a $C^{2}$ function $y_{n+1}=\psi(y)$, 
%then 
%for $Y\in \Sigma$ using the expresion $Y=X+s(v,X)(-Qv,Qv+\kappa)$, 
Then we can write $y_{n+1}=x_{n+1}+s_X(\Lambda(v))(Q(v)+k)=\psi(x-s_X(\Lambda(v))\,Q(v)v)$.
We therefore get that $H$ satisfies the implicit equation
%$H(v,X)=s(v,X)Q(v)$ we obtain as implicit definition of $H$ by the formula 
\begin{equation}\label{eq;implicitformulaforH}
x_{n+1}+H(v,X)\left(\frac{Q(v)+\kappa}{Q(v)}\right)=\psi(x-H(v,X)v).
\end{equation}

Proceeding as in Subsection \ref{subsec:localandglobalrefractors}, setting $G=1/H$ 
we need to check that
\begin{equation}\label{eq:negativequantity}
\langle D^{2}G(0,X)\xi,\xi\rangle<0
\end{equation} 
for all $|\xi|=1$.
We will prove \eqref{eq:negativequantity} for $X=0$ and by continuity \eqref{eq:negativequantity} will hold for $X$ in a neighborhood of $X=0$.
Indeed, we next compute $D_{ij}G(0,0)$ in terms of $\psi$.
Since $G=1/H$, we first proceed to calculate the derivatives of $H$. From \eqref{eq:defofLambdav}, $Q(v)=\dfrac{\sqrt{1+(1-\kappa^2)|v|^2}-\kappa}{1+|v|^2}$ and so $Q(0)=1-\kappa$, $D_{i}Q(0)=0$ and $D_{ij}Q(0)=-\delta_{ij}(1-\kappa)^{2}$.
%THIS LAST ONE IS WRONG, I GET $D_{ij}Q(0)=-\delta_{ij}(1-\kappa)$.
From \eqref{eq;implicitformulaforH}, we get $H(0,0)=(1-\kappa)\psi(0)$, $D_{i}H(0,0)=-(1-\kappa)^{2}D_{i}\psi(0)\,\psi(0)$, and 
\[
D_{ij}H(0,0)=
(1-\kappa)^{3}\,\psi(0)\left(\frac{-\kappa}{1-\kappa}\delta_{ij}+2\,D_{i}\psi(0)\,D_{j}\psi(0)+D_{ij}\psi(0)\,\psi(0)\right).
\]
%CHECK!
%I GET MINUS IN $D_{i}H(0,0)$
%I ALSO GET A DIFFERENT EXPRESSION FOR $D_{ij}H(0,0)$,
%I HAVE CHANGED IT.
%
Now, noticing that $D_{ij}G=H^{-3}(2D_{i}H\,D_{j}H-H\,D_{ij}H)$ and inserting the above expressions yields 

\[
D_{ij}G(0,0)=\frac{1-\kappa}{\psi(0)}\left(\frac{\kappa}{1-\kappa}\,\delta_{ij}-\psi(0)\,D_{ij}\psi(0)\right).
\]
%CHECK! I GET A DIFFERENT CONSTANT IN $\kappa$
We therefore obtain that 
\begin{align*}
\langle D^{2}G(0,0)\xi,\xi\rangle
&=\frac{1-\kappa}{\psi(0)}\left(\frac{\kappa}{1-\kappa}-\psi(0)\langle D^{2}\psi(0)\xi,\xi\rangle\right).
\end{align*}

Thus \eqref{eq:negativequantity} is equivalent to
\begin{equation}\label{eq:quantitativeconditiononthetarget}
\dfrac{\kappa}{1-\kappa}<\psi(0)\,\langle D^{2}\psi(0)\xi,\xi\rangle,
\end{equation}
for all unit vectors $\xi$.
%\footnote{If we assume $\xi\cdot \eta=0$, then instead of \eqref{eq:quantitativeconditiononthetarget}
%we obtain $\kappa<\psi(0)^2\,\langle D^{2}\psi(0)\xi,\xi\rangle$.}.
%THIS CONDITION IS DIFFERENT FROM THE ONE BEFORE.

We have then proved that for $\Sigma$ the graph of $\psi$, and for $X_0=(0,0)$,$Y_0=(0,\psi(0))$, the condition $\dfrac{d^{2}}{d\epsilon^{2}}\langle D^{2}\phi(x_0,Y_\epsilon,X_0)\eta,\eta\rangle|_{\epsilon=0}<0$ holds for all unit vectors $\xi\perp \eta$, provided \eqref{eq:quantitativeconditiononthetarget} holds.
%$\frac{3k+2}{1-k}<\psi(0)\langle D^{2}\psi(0)\xi,\xi\rangle)$ for all unit vectors $\xi$.
By continuity we have that  $\dfrac{d^{2}}{d\epsilon^{2}}\langle D^{2}\phi(x,Y(\epsilon),X)\eta,\eta\rangle|_{\epsilon=0}\leq -C|\xi|^{2}|\eta|^{2}$
for all vectors $\xi\perp \eta$ and for all $X$ close to $X_0$ and $Y$ close to $Y_0$.
%CLARIFY THIS LAST PART??
%Finally, condition 
%\eqref{eq:thirdderivatives} follows since $\phi\in C^2$.

\begin{remark}\label{rmk:horizontalplaneisnotarefractor}\rm

We show that if the target $\Sigma$ is a horizontal plane, then \eqref{eq:convexityconditiononthetargetbisbis} does not hold when $n\geq 2$. Indeed, we show it is not true that $\phi(x,Y(\lambda),X_0)\geq \min \{\phi(x,\bar Y,X_0),
 \phi(x,\hat Y,X_0)\}$.
 In $\R^{3}$, let $\bar Y=(0,-1,0)$, $\hat Y=(0,1,0)$, $\bar 
 \phi(x,y)=-\dfrac{kb}{1-k^{2}}-\left(\dfrac{b^{2}}{(1-k^{2})^{2}}-\dfrac{x^{2}+(y+1)^{2}}{1-k^{2}}\right)^{1/2}$ and $\hat\phi(x,y)=-\dfrac{kb}{1-k^{2}}-\left(\dfrac{b^{2}}{(1-k^{2})^{2}}-\dfrac{x^{2}+(y-1)^{2}}{1-k^{2}}\right)^{1/2}$, with $b$ sufficiently large.
 Let $X_0=(0,0,\bar\phi(0,0))=(0,0,\hat\phi(0,0))$ and we have
 $\bar \phi(x,y)=\phi((x,y),\bar Y, X_0)$ $\hat \phi(x,y)=\phi((x,y),\hat Y, X_0)$.
We also have that $\Lambda(v(\lambda))$ is on the two dimensional plane $x=0$, and so $[\bar Y,\hat Y]_{X_0}=[\bar Y,\hat Y]$ is the straight segment from $\bar Y$ to $\hat Y$. 
We have $Y(1/2)=(0,0,0)$. Let $\phi_0(x,y):=\phi((x,y),Y(1/2),X_0)=
-\dfrac{kb_0}{1-k^{2}}-\left(\dfrac{b_0^{2}}{(1-k^{2})^{2}}-\dfrac{(x^{2}+y^{2})}{1-k^{2}}\right)^{1/2}$, where $b_0$ is chosen so that $\phi_0(0,0)=\bar\phi(0,0))=\hat\phi(0,0)$. That is, $b_0=\dfrac{kb+(b^{2}-(1-k^{2}))^{\frac{1}{2}}}{1+k}$. We claim that $\phi_0(x,0)<\bar\phi(x,0)=\hat\phi(x,0)$ for $x\neq 0$, $x$ small. 
Let $g(x)=\phi_0(x,0)$ and $h(x)=\bar\phi(x,0)$. Then one can check that $(g-h)^{\prime}(0)=0$ and $(g-h)^{\prime\prime}(0)<0$ which gives the claim. 

%CHECK THE COMPUTATION!!!!! 

%INCLUDE PICTURE OF THE THREE ELLIPSOIDS, TAKE $b=2$ AND DETERMINE $b_0$.
 \end{remark}

\setcounter{equation}{0}
\section{On the definition of refractor}\label{sec:alternativedefinitionofrefractor}
%REVISE

We can define refractor with ellipsoids touching $u$ from below, that is, the ellipsoids enclose $u$.
%$u$ ENCLOSES THE ELLIPSOIDS?????
In fact, we can define analogously to \eqref{eq:definitionofrefractormapping}
\begin{equation}\label{eq:definitionofrefractormappingbis}
\tilde F_u(x_0)=\{Y\in \Sigma: \text{$u(x)\geq \phi(x,Y,X_0)$ for all $x\in\Omega$}\},
\end{equation}
and we can say $u$ is a {\it refractor} if $\tilde F_u(x_0)\neq \emptyset$ for all $x_0\in \Omega$.
With this new definition, we can obtain the same regularity results as with definition \eqref{eq:definitionofrefractormapping} by changing the inequalities accordingly.
We indicate the changes. 
In Definition \ref{def:localtargetbis}, condition \eqref{eq:convexityconditiononthetargetbisbis} is replaced by:
\begin{equation}\label{eq:convexityconditiononthetargetbis}
\phi(x,Y_{\bar X}(\lambda),Z)\leq\max\left\{\phi(x,\bar Y,Z),\phi(x,\hat Y,Z)\right\}-C_1\,|\bar Y-\hat Y|^{2}|x-z|^{2}.
\end{equation}
%for all $x\in\Omega$, $0\leq \lambda \leq 1$ with $|x-x_0|\leq C_2$. 
Inequality \eqref{eq:AWP} is replaced by:
\begin{equation}\label{eq:AWbisbis}
\dfrac{d^{2}}{d\epsilon^{2}}\langle D^{2}\phi(x_0,Y_\epsilon,X_0)\eta,\eta\rangle|_{\epsilon=0}\geq C|\xi|^{2}|\eta|^{2}.
\end{equation}
Lemma \ref{lm:estimateconvexcombinationrefractor} is replaced by the convexity of $u$.
The inequality \eqref{eq:mainlemmainequality} is replaced by:
\begin{equation}\label{eq:mainlemmainequalitybis}
u(x)-\phi(x,Y,X_0^{\star})\geq -C_0|\bar Y-\hat Y||\bar x-\hat x|-C_1|Y(\lambda)-Y||x-x_0|+C_2\,|\bar Y-\hat Y|^{2}|x-x_0|^{2},
\end{equation}
and in the proof of Lemma \ref{lm:mainlemmarefractor}, $\min$ is replaced by $\max$ with the corresponding changes in the inequalities.
For the example in Section \ref{sec:example}, we now get that condition \eqref{eq:quantitativeconditiononthetarget}
is replaced by
\begin{equation}\label{eq:quantitativeconditiononthetargetbis}
\dfrac{\kappa}{1-\kappa}>\psi(0)\,\langle D^{2}\psi(0)\xi,\xi\rangle,
\end{equation}
for all unit vectors $\xi$. 
%IS THIS LAST CONDITION RIGHT?

%
\setcounter{equation}{0}
\section{Regularity results for the parallel reflector problem in the near field case}\label{sec:regularityforthereflectorproblem}
%REVISE

In this section we shall prove results that are similar to the ones proved in the previous sections but for the reflector problem.
Since the arguments are similar, we will omit most details.

\subsection{Reflection}\label{subsec:reflectionprocess}
We first review the process of reflection. Our setting is $\R^{n+1}$, and points will be denoted by $X=(x,x_{n+1})$. 
We consider parallel rays moving in the direction $e_{n+1}$. Let $T$  be a hyperplane in $\R^{n+1}$ with upper unit normal $N$ and let $X\in T$. By Snell law of reflection, a ray coming from below with direction $e_{n+1}$ that hits $T$ at $X$ is reflected in the unit direction $\Lambda=e_{n+1} -2\,(e_{n+1}\cdot N)\,N$. In particular, if $v\in \R^{n}$ and $N=\dfrac{(-v,1)}{(1+|v|^{2})^{\frac{1}{2}}}$, then the reflected direction is the unit vector
\begin{equation}\label{eq:reflectedrayLambda}
\Lambda (v)=\left(\frac{2v}{1+|v|^{2}},\frac{|v|^{2}-1}{1+|v|^{2}}\right).
\end{equation}
The reflected ray consists of the points $Y=X+s\Lambda$, for $s>0$.
We have in mind here that $v=Du(x)$ and $X=(x,u(x))$, where $u$ is a reflector.

If $b>0$, and $Y\in \R^{n+1}$, then the set of $X\in \R^{n+1}$ with $|X-Y|+x_{n+1}-y_{n+1}=b$ is a downwards paraboloid with focus at $Y$. It can be written as the graph of the function 
$p(x,Y)=y_{n+1}+\dfrac{b^{2}-|x-y|^{2}}{2b}$.
The ray with direction $e_{n+1}$ that hits the graph of $p$ at $X=(x,p(x,Y))$ is reflected in direction $Y-X$. 
If $Y,X_0\in \R^{n+1}$ with $X_0$ not in the vertical ray with direction $-e_{n+1}$ emanating from $Y$, then there exists a unique paraboloid
 with focus at $Y$ passing through $X_0$.
Such a paraboloid is described by the function $p(x,Y,X_0)=y_{n+1}+\dfrac{b^{2}-|x-y|^{2}}{2b}$ where $b=|X_0-Y|+x_{0_{n+1}}-y_{n+1}$. 
%IS THIS LAST EXPRESSION USED?
We will frequently use the following fact:
\begin{align}\label{eq:frecuentfactusedparaboloids} 
\text{if the focus $Y$ of the paraboloid defined by $p(x,Y,X_0)$ satisfies}\\
\text{$Y=X_0+s\Lambda(v)$ for some $s>0$ and $v\in \R^n$, then $D_xp(x_0,Y,X_0)=v$,}\notag
\end{align}
where $\Lambda(v)$ is given by \eqref{eq:reflectedrayLambda}.

\subsection{Set up}\label{subsection:setupreflector}

Let $\Omega\subset \R^{n}$ be open and bounded, and let $C_\Omega$ be the cylinder $C_{\Omega}=\Omega\times [0,M]$. 
For a fixed number $\beta>0$
we define the region $$\mathcal T=\{Y\in R^{n+1}:|X-Y|+x_{n+1}-y_{n+1}\geq\beta \text{ for all }X\in C_{\Omega}\}.$$
The set $\mathcal T$ consists of the points $Y$ such that the cylinder $C_\Omega$ is contained outside the interior of the 
paraboloid $|X-Y|+x_{n+1}-y_{n+1}=\beta$.
We will assume that the target $\Sigma$ has convex hull bounded and contained in $\mathcal T$.

Proceeding as in Subsection \ref{subsec:estimatesderivativesofphi}, it is easy to see that
$\left| \dfrac{\partial p}{\partial x_{0_{n+1}}}(x,Y,X_0)\right|$, 
$\left| \dfrac{\partial^2 p}{\partial x_i\partial y_j}(x,Y,X_0)\right|$
are bounded uniformly for all $x\in \Omega$, $Y\in K\subset \mathcal T$, and $X_0\in C_\Omega$ for $1\leq i\leq n$ and
$1\leq j\leq n+1$, where $K$ is compact.
Hence as in Lemmas \ref{lm:lipschitzinX0} and \ref{lm:estimateslipinYandx}, we obtain the following.

\begin{lemma}\label{lm:lipschitzcontinuityofparaboloids}
If $X_0\in C_{\Omega}$ with $X_0+he_{n+1}\in C_\Omega$, and $Y,\bar Y\in\Sigma$, then we have
%The following estimates hold:
$$|p(x,Y,X_0)-p(x,Y,X_0+he_{n+1})|\leq C|h|,$$ and 
$$|p(x,Y,X_0)-p(x,\bar Y,X_0)|\leq C|x-x_0||Y-\bar Y|,$$ 
for all $x\in \Omega$.
%for all $x\in\Omega$, for all $X_0\in C_{\Omega}$ and for all $\bar Y,\hat Y\in\Sigma$.
%CHECK LIKE LEMMA FOR ELLIPSOIDS.
\end{lemma}

\subsection{Hypothesis on the target set}\label{subsect:hypothesesontargetrflector}

Given $\bar Y,\hat Y\in \Sigma$, $X_0\in C_{\Omega}$, we let $\bar v=D_xp(x_0,\bar Y,X_0)$, $\hat v=D_xp(x_0,\hat Y,X_0)$, and $v(\lambda)=(1-\lambda)\bar v +\lambda \hat v$, for $\lambda \in [0,1]$.
Consider the set of points 
$$C(X_0,\bar Y,\hat Y)=\{Y=X_0+s\Lambda(v(\lambda)):s>0,\lambda\in [0,1]\},$$
where $\Lambda$ is given by \eqref{eq:reflectedrayLambda}.
%MAKE PICTURE?
Notice that from \eqref{eq:frecuentfactusedparaboloids}, if $Y=X_0+s\Lambda(v(\lambda))$, then $D_xp(x_0,Y,X_0)=v(\lambda)$. 
%We will denote the curve in Lemma \ref{lm:convexityoffociofellipsoids} joining $\bar Y$ and $\hat Y$ with the notation $\{\bar Y,\hat Y\}_{X_0}$; this means that if $Y\in \{\bar Y,\hat Y\}_{X_0}$, then $Y=Y^*(\lambda)=X_0+s(\lambda)\Lambda(v(\lambda))$ for some $\lambda\in [0,1]$ and $p(x,Y^*(\lambda),X_0)=(1-\lambda)p(x,\bar Y,X_0)+\lambda p(x,\hat Y,X_0)$ for all $x$.

For $X_0\in C_{\Omega}$ and for $\bar Y,\hat Y\in\Sigma$ we assume $[\bar Y,\hat Y]_{X_0}:=C(X_0,\bar Y,\hat Y)\cap\Sigma$ is a curve joining $\bar Y$ and $\hat Y$. 
%Notice that both curves $[\bar Y,\hat Y]_{X_0}, \{\bar Y,\hat Y\}_{X_0}$ are contained in the set $C(X_0,\bar Y,\hat Y)$.

We introduce the following condition on the target $\Sigma$, similar to Definition \ref{def:localtargetbis} for refractors.
\begin{definition}\label{def:reflectorlocaltargetbis}
If $X_0\in C_\Omega$ we say that the {\it target $\Sigma$ is regular from $X_0$} if 
there exists a neighborhood $U_{X_0}$ and positive constants $C_{X_0}$, depending on $U_{X_0}$, such that for all $\bar Y,\hat Y\in \Sigma$ and $Z=(z,z_{n+1})\in U_{X_0}$
we have
\begin{equation}\label{eq:reflectorconvexityconditiononthetargetbisbis}
\max\left\{p(x,\bar Y,Z),p(x,\hat Y,Z)\right\}\geq p(x,Y_{Z}(\lambda),Z)+C_{X_0}\,|\bar Y-\hat Y|^{2}|x-z|^{2}\end{equation}
for all $x\in \Omega$, $1/4\leq \lambda \leq 3/4$, and $Y_{Z}(\lambda)=Z+s_{Z}(\Lambda(v(\lambda)))\,\Lambda(v(\lambda))$. 
%SIMILAR TO A STRONG, CHECK MCCANN-FIGALLI-RECENT PAPER.
%Where $C$ stands for possibly different structural constants.
%THE CONSTANTS C DEPEND ON THE POINTS OR ARE UNIFORM???
Here $\bar v=D_x\phi(z,\bar Y,Z)$, $\hat v=D_x\phi(z,\hat Y,Z)$,
 and $v(\lambda)=(1-\lambda)\bar v +\lambda \hat v$.
 \end{definition}

As in the case of refractors, we also have a differential condition that is equivalent to \eqref{eq:reflectorconvexityconditiononthetargetbisbis}. This is the contents of the following theorem. 
 
\begin{theorem}\label{thm:awforperpendicularvectorsreflectors}
Suppose that there exists a constant $C$ such that for all $\xi$ and $\eta$, perpendicular vectors in $R^{n}$, and for  $X_0\in C_{\Omega}$ and for $Y_0\in\Sigma$, we have 
\begin{equation}\label{eq:AWPreflectors}
\dfrac{d^{2}}{d\epsilon^{2}}\left. \left\langle D_x^{2}p(x_0,Y_\epsilon,X_0)\eta,\eta\right\rangle \right|_{\epsilon=0}\geq C|\xi|^{2}|\eta|^{2},
\end{equation}
where, $v_0=Dp(x_0,Y_0,X_0)$ and $Y_\epsilon=X_0+s_{X_0}(\Lambda(v_0+\epsilon\xi))\Lambda(v_0+\epsilon\xi)$.

Then there exists a structural constant $C$ such that for $\bar Y,\hat Y\in \Sigma$ and $X_0\in C_{\Omega}$ we have for $\lambda\in[1/4,3/4]$ and for all $x\in \Omega$ that
\begin{equation}\label{eq:minconditiontargetreflectors}
\max\left\{p(x,\bar Y,X_0),p(x,\hat Y,X_0)\right\}\geq p(x,Y_{X_0}(\lambda),X_0)+C\,|\bar Y-\hat Y|^{2}|x-x_0|^{2}.
\end{equation}

Conversely, \eqref{eq:minconditiontargetreflectors} implies \eqref{eq:AWPreflectors}.
\end{theorem}
 
\begin{proof}
That \eqref{eq:minconditiontargetreflectors} implies \eqref{eq:AWPreflectors}, follows in the same way as
\eqref{eq:minconditiontarget} implies \eqref{eq:AWP} in Theorem \ref{thm:awforperpendicularvectors}.

We first show that if \eqref{eq:AWPreflectors} holds for $\xi\perp \eta$, then it holds for all vectors $\xi,\eta$.
In fact, we have
\begin{equation}\label{eq:formulaforinnerprodofp}
\left\langle D_x^{2}p(x_0,Y_\epsilon,X_0)\eta,\eta\right\rangle
=
-\dfrac{|\eta|^2}{2}\, \dfrac{1+|v_0+\epsilon\,\xi|^2}{s_{X_0}\left( \Lambda(v_0+\epsilon\,\xi)\right)}.
\end{equation}
So if $\xi\cdot \eta\neq 0$, we pick $\eta'$ with $|\eta'|=|\eta|$ and $\eta'\cdot \xi=0$, and we have
$\left\langle D_x^{2}p(x_0,Y_\epsilon,X_0)\eta,\eta\right\rangle=\left\langle D_x^{2}p(x_0,Y_\epsilon,X_0)\eta',\eta'\right\rangle$.
Let $\bar Y, \hat Y\in \Sigma$ and $X_0\in C_\Omega$, $\bar v=Dp(x_0,\bar Y, X_0)$, $\hat v=Dp(x_0,\hat Y, X_0)$,
$\xi =\hat v-\bar v$, and $Y_{X_0}(\lambda)=X_0+s_{X_0}\left( \Lambda (\bar v+\lambda \xi)\right)\,\Lambda (\bar v+\lambda \xi)$.
We then have
\[
\dfrac{d}{d\lambda^2}\left\langle D_x^2 p\left(x_0, Y_{X_0}(\lambda), X_0 \right)\,\eta,\eta \right\rangle \geq C\,|\xi|^2|\eta|^2,
\]
for $0\leq \lambda \leq 1$.
Fix $x\in \Omega$, and let $f(\lambda)=\left\langle D_x^2 p\left(x_0, Y_{X_0}(\lambda), X_0 \right)\,(x-x_0),(x-x_0) \right\rangle$.
Since $f''(\lambda)\geq C\,|\xi|^2|x-x_0|^2\geq C\,|\bar Y-\hat Y|^2|x-x_0|^2$, where the second inequality follows from the analogue of \eqref{eq:conditionbetweenYandv} for reflectors, it follows that
$(1-\lambda)f(0)+\lambda f(1)\geq f(\lambda)+C\,\lambda \,(1-\lambda)\,|x-x_0|^2 |\bar Y-\hat Y|^2$.
Therefore,
\begin{align*}
&\max\left\{p(x,\bar Y,X_0),p(x,\hat Y,X_0)\right\}-p(x, Y_{X_0}(\lambda), X_0)\\
&\geq 
(1-\lambda)\,p(x,\bar Y, X_0)+\lambda \,p(x,\hat Y,X_0)-p(x,Y_{X_0}(\lambda),X_0)\\
&=
\frac12 \left\langle \left((1-\lambda)\,D_x^2p(x_0,\bar Y,X_0)+\lambda \,D_x^2p(x_0,\hat Y,X_0) -
\,D_x^2p(x_0,Y_{X_0}(\lambda),X_0)\right)(x-x_0),x-x_0\right\rangle\\
&\geq 
C\,\lambda \,(1-\lambda) \,|x-x_0|^2 |\bar Y-\hat Y|^2,
\end{align*}
for all $x\in \Omega$, and $0\leq \lambda \leq 1$.
This completes the proof of the theorem.

\end{proof}

 \begin{remark}\label{rmk:localequivalenceperpvectorsreflectors}\rm
 Analogously to Remark \ref{rmk:localequivalenceperpvectors}, a  
 local version of \eqref{eq:AWPreflectors} can be stated as follows: The target $\Sigma$ is regular from $X_0\in C_{\Omega}$ if there exists a neighborhood $U_{X_0}$ and a constant $C$ depending on $X_0$ such that for all $Y_0\in\Sigma$ and for all $Z\in U_{X_0}$ and for all vectors $\xi$ and $\eta$ such that $\xi\perp \eta$ we have 
\begin{equation}\label{eq:AWPLreflectors}
\dfrac{d^{2}}{d\epsilon^{2}}\left. \left\langle D_x^{2}p(z,Y_\epsilon,Z)\eta,\eta\right\rangle \right|_{\epsilon=0}\geq C\,|\xi|^{2}|\eta|^{2}.
\end{equation}
where $Y_{\epsilon}=Z+s(\Lambda(v+\epsilon\xi))\Lambda(v+\epsilon\xi)$ and $v=Dp(z,Y_0,Z)$.

Following the proof of Theorem \ref{thm:awforperpendicularvectorsreflectors}, one can show that \eqref{eq:AWPLreflectors} is equivalent to \eqref{eq:reflectorconvexityconditiononthetargetbisbis}.
% for $\lambda\in [1/4,3/4]$.
\end{remark}

\subsection{Definition of parallel reflector and hypothesis on the measures}\label{subsect;definitionofreflectorandhypotheses}

We say $u:\Omega\rightarrow [0,M]$ is a parallel reflector from $\Omega$ to $\Sigma$ if for each $x_0\in\Omega$,
there exists $Y\in\Sigma$ such that $u(x)\geq p(x,Y,X_0)$ for all $x\in\Omega$, where $X_0=(x_0,u(x_0))$. In this case, 
we say $Y\in F_u(x_0)$. 
%PROVE EXISTENCE OF SOLUTIONS WITH THIS DEFINITION (DIFFERENT FROM WANG's DEFINITION).
Any reflector is Lipschitz in $\Omega$ with a uniform Lipschitz constant depending on the bounds for the derivatives of $p$, which are uniform for $Y\in K\Subset \mathcal T$, $x\in \Omega$ and $X_0\in C_\Omega$. 

Existence of solutions with this definition of parallel reflector can be proved in a way similar to the existence of parallel refractors as done in \cite{gutierrez-tournier:parallelrefractor}. We omit the corresponding details.

We make the following hypothesis on the measures.
%If $\sigma$ denotes the measure on the target $\Sigma$,  
%we assume that $\sigma\left(F_u(B_{\eta})\right)\leq C\eta^{n/q}$ for all balls $B_{\eta}\subseteq\Omega$,
%see \eqref{eq:assumptiononmeasuresigma}.
%THIS CONDITION DEPENDS ON $u$, INCLUDE THE CONDITION INSIDE THE THEOREM.

Also similarly to \eqref{eq:conditiononsigmameasureofthetube}, 
we introduce the following local condition at $X_0\in C_\Omega$ between the measure $\sigma$ and target $\Sigma$:
%We assume the following condition:
There exist a neighborhood $U_{X_0}$ and a constant $\hat C>0$ depending on $X_0$ such that
\begin{equation}\label{eq:conditiononsigmameasureofthetubereflectors}
\sigma\left(N_{\mu}\left(\left\{[\bar Y,\hat Y]_{Z}:\lambda\in [1/4,3/4]\right\}\right)\cap\Sigma \right)\geq \hat C\, \mu^{n-1}\,|\bar Y-\hat Y|
\end{equation} 
for any $\bar Y,\hat Y\in\Sigma$, $Z\in U_{X_0}$ and for all $\mu>0$ small (depending on $X_0$). Here $N_{\mu}(E)$ denotes the $\mu$- neighborhood of the set $E$ in $R^{n+1}$.

%we assume that $\sigma(N_{\mu}(\{[\bar Y,\hat Y]_{X_0}:\lambda\in [\frac{1}{4},\frac{3}{4}]\})\cap\Sigma)\geq C \mu^{n-1}|\bar Y-\hat Y|$ for any $\bar Y,\hat Y\in\Sigma$, $X_0\in\Omega$ and for all $\mu>0$ small,  where $N_{\mu}$ denotes the $\mu$- neighborhood in $R^{n+1}$.

The following lemma is the analogue of Lemma \ref{lm:estimateconvexcombinationrefractor}.

\begin{lemma}\label{lm:estimateconvexcombinationreflector}

If $u$ is a parallel reflector, then for any $\bar x,\hat x\in\Omega$ and for any $s\in [0,1]$ we have $u((1-s)\bar x+s\hat x)\leq (1-s)u(\bar x)+su(\hat x)+C|\bar x-\hat x|^{2}s(1-s)$, where $C$ is a structural constant.
\end{lemma}

%\textbf{proof}
%
%The same way as in the refractor problem.

The following lemma is the analogue of Lemma \ref{lm:mainlemmarefractor}.
\begin{lemma}\label{lm:mainlemmareflector}

Let $u$ be a parallel reflector such that the target $\Sigma$ is regular from $X^{\star}=(x^{\star},u(x^{\star}))$ in the sense of Definition \ref{def:reflectorlocaltargetbis}. 
There exist constants $\delta$ and $C_1$, depending on $X^{\star}$ such that if $\bar x,\hat x\in B_{\delta}(x^{\star})
\subset \Omega$,  $\bar Y\in F_u(\bar x),\hat Y\in F_u(\hat x)$ and $|\bar Y-\hat Y|\geq |\bar x-\hat x|$,
%I THINK HERE WE NEED TO ASSUME $$|\bar Y-\hat Y|\geq \alpha\, |\bar x-\hat x|$$ FOR SOME $\alpha>0$ BECAUSE OF THE NEXT THEOREM.
then there exists $x_0\in \overline{\bar x\hat x}$ such that if $X_0^{\star}=(x_0,u(x_0))$, then
%Let $\bar x,\hat x\in\Omega$ and $\bar Y\in F_u(\bar x),\hat Y\in F_u(\hat x)$ and assume that $|\bar Y-\hat Y|\geq |\bar x-\hat x|$. Then there exists $x_0\in [\bar x,\hat x]$ such that setting $X_0^{\star}=(x_0,u(x_0))$ 
we have for all $Y(\lambda)\in [\bar Y,\hat Y]_{X_0^{\star}}$ and for all $Y\in\Sigma$ and for all $x\in\Omega$ that the following inequality holds 

$u(x)-p(x,Y,X_0^{\star})\geq - C|\bar Y-\hat Y||\bar x-\hat x|-C|Y(\lambda)-Y||x-x_0|+C_1\lambda(1-\lambda)|\bar Y-\hat Y|^{2}|x-x_0|^{2}$.
We remark that $C$ is a structural constant.
\end{lemma}

\begin{proof}

The proof is very much the same as in the refractor problem. We indicate the main points.

Since $\Sigma$ is regular from $X^{\star}$, there exists a neighborhood $U_{X^{\star}}$ of $X^{\star}$ such that 
\eqref{eq:reflectorconvexityconditiononthetargetbisbis} holds for all $\bar Y,\hat Y\in \Sigma$ and all $Z\in U_{X^{\star}}$.
Since parallel reflectors are uniformly Lipschitz in $\Omega$, there exists $\delta>0$ such that $(x,u(x))\in U_{X^{\star}}$ for all $x\in B_\delta(x^\star)$.
For $x\in\Omega$ we have $u(x)\geq \max\{p(x,\bar Y,\bar X),p(x,\hat Y,\hat X)\}$, and there exists $x_0\in [\bar x,\hat x]$ such that $p(x_0,\bar Y,\bar X)=p(x_0,\hat Y,\hat X):=x_{0_{n+1}}$.  Set $X_0=(x_0,x_{0_{n+1}})$, and $X_{0}^{\star}=(x_0,u(x_0))\in U_{X^{\star}}$ and notice that $u(x_0)\geq x_{0_{n+1}}$. Similarly as in Lemma \ref{lm:mainlemmarefractor} we get 
\begin{align*}
u(x)&\geq \max\{p(x,\bar Y,\bar X),p(x,\hat Y,\hat X)\}=\max\{p(x,\bar Y,X_0),p(x,\hat Y,X_0)\}\\
&=\max\{p(x,\bar Y,X_{0}^{\star}),p(x,\hat Y,X_{0}^{\star})\}-E.
 \end{align*}
%where $E$ is defined by equality.

Using Lemmas \ref{lm:lipschitzcontinuityofparaboloids} and \ref{lm:estimateconvexcombinationreflector}, and proceeding as in the proof of claim \eqref{eq:claimmainlemma}, we get $0\leq E\leq C\left(u(x_0)-x_{0_{n+1}}\right)\leq C|\bar Y-\hat Y||\bar x-\hat x|$. Hence $\max\{p(x,\bar Y,X_{0}^{\star}),p(x,\hat Y,X_{0}^{\star})\}-E\geq \max\{p(x,\bar Y,X_{0}^{\star}),p(x,\hat Y,X_{0}^{\star})\}-C|\bar Y-\hat Y||\bar x-\hat x|\geq p(x,Y_{X_0^\star}(\lambda),X_{0}^{\star})+C_1\lambda(1-\lambda)|\bar Y-\hat Y|^{2}|x-x_0|^{2}-C|\bar Y-\hat Y||\bar x-\hat x|$, where we have used \eqref{eq:reflectorconvexityconditiononthetargetbisbis}.
%Lemma \ref{lm:convexityoffociofellipsoids} and inequality \eqref{eq:inequalityconsequenceofessentialhypothesisreflector}. 
Now using the second inequality in Lemma \ref{lm:lipschitzcontinuityofparaboloids}, we get $p(x,Y_{X_0^\star}(\lambda),X_{0}^{\star})+C_1\lambda(1-\lambda)|\bar Y-\hat Y|^{2}|x-x_0|^{2}-C|\bar Y-\hat Y||\bar x-\hat x|\geq p(x,Y,X_{0}^{\star})-C|x-x_0||Y-Y(\lambda)|+C_1\lambda(1-\lambda)|\bar Y-\hat Y|^{2}|x-x_0|^{2}-C|\bar Y-\hat Y||\bar x-\hat x|$, which proves the lemma.
\end{proof}

We then obtain results similar to Theorems \ref{thm:mainestimate} and  \ref{thm:holdercontinuity} for parallel reflectors. We remark that from Theorem \ref{thm:awforperpendicularvectorsreflectors}, inequality 
\eqref{eq:AWPLreflectors} implies \eqref{eq:reflectorconvexityconditiononthetargetbisbis}
for all $x\in \Omega$, and therefore for parallel reflectors the argument in Section \ref{subsec:localandglobalrefractors} is not needed.
\begin{theorem}\label{thm:reflectorsmainestimate}
Suppose $u$ is a parallel reflector, the target $\Sigma$ is regular from $X^{\star}=(x^{\star},u(x^{\star}))$ in the sense of Definition \ref{def:reflectorlocaltargetbis}.
%, and \eqref{eq:curveisaconnectedset} holds. 
There exist a ball $B_\delta(x^{\star})\subset \Omega$, and a
constant $M>0$ depending on $X^{\star}$ such that if $\hat x,\bar x\in B_{\delta/2}(x^{\star})$, $\bar Y\in F_u(\bar x),\hat Y\in F_u(\hat x)$ are such that
\begin{equation}\label{eq:differenceofYbiggerthandifferenceofxwithdistancereflectors}
\dfrac{|\bar Y-\hat Y|}{|\bar x-\hat x|}\geq \max\left\{1, \left( \dfrac{2\,M}{\delta}\right)^2 \right\} ,
\end{equation}
then there exists $x_0\in \overline{\bar x,\hat x}$ such that we have 
\[
N_{\mu}\left(\left\{Y(\lambda)\in [\bar Y,\hat Y]_{X_0^{\star}}:\lambda\in [1/4,3/4]\right\}\right)\cap\Sigma\subseteq F_u(B_{\eta}(x_0)),
\] 
with $X_0^{\star}=(x_0,u(x_0))$, $\mu=|\bar Y-\hat Y|^{\frac{3}{2}}|\bar x-\hat x|^{\frac{1}{2}}$ and $\eta=M\,\dfrac{|\bar x-\hat x|^{\frac{1}{2}}}{|\bar Y-\hat Y|^{\frac{1}{2}}}$.
\end{theorem}

\begin{theorem}\label{thm:holdercontinuityreflectors}
 Suppose $u$ is a parallel reflector, the target $\Sigma$ is regular from $X^{\star}=(x^{\star},u(x^{\star}))$ in the sense of Definition \ref{def:reflectorlocaltargetbis}, and
 there exist constants $C_0>0$ and $1\leq q<\dfrac{n}{n-1}$ such that
\begin{equation}\label{eq:assumptiononmeasuresigmareflectors}
\sigma\left(F_u(B_{\eta})\right)\leq C_0\, \eta^{n/q}
\end{equation}
for all balls $B_{\eta}\subseteq\Omega$. Suppose in addition that the local condition \eqref{eq:conditiononsigmameasureofthetubereflectors} is satisfied at $X^\star$. 

Then there exist a ball $B_\delta(x^{\star})\subset \Omega$, and 
a constant $M>0$ depending on $X^{\star}$, such that if $\hat x,\bar x\in B_{\delta/2}(x^{\star})$, $\bar Y\in F_u(\bar x),\hat Y\in F_u(\hat x)$ are such that
\begin{equation}\label{eq:differenceofYbiggerthandifferenceofxwithdistancebisreflectors}
\dfrac{|\bar Y-\hat Y|}{|\bar x-\hat x|}\geq \max\left\{1, \left( \dfrac{2\,M}{\delta}\right)^2\right\} ,
\end{equation}
%and the target $\Sigma$ verifies \eqref{eq:convexityconditiononthetarget}, 
%$\sigma$ satisfies \eqref{eq:conditiononsigmameasureofthetube}, and 
%.  
%and the measure $\mu$ has density $f\in L^p(\Omega)$ for some $n<p\leq \infty$.
%Let $u$ be a refractor.Let $\bar x,\hat x\in\Omega$ and $\bar Y\in F_u(\bar x),\hat Y\in F_u(\hat x)$ and assume that $\frac{\epsilon^{2}}{C_4^{2}}|\bar Y-\hat Y|\geq |\bar x-\hat x|$. 
then we have $|\bar Y-\hat Y|\leq C_1\,|\bar x-\hat x|^\alpha$ with $\alpha=\dfrac{\dfrac{n}{2q}-\dfrac{n-1}{2}}{1+\dfrac32 (n-1)+\dfrac{n}{2q}}$, where $C_1$ depends only on $C_0$ and $\hat C$ in \eqref{eq:conditiononsigmameasureofthetubereflectors}, and therefore from $X^\star$.
\end{theorem}

We can then proceed in exactly the same way as in the refractor problem to get the analogues of Theorem \ref{thm:c1alpha},
and Corollary \ref{corollary:applicationtorefractors} for parallel reflectors.

%\begin{theorem}
%
%If $u$ is a parallel reflector, then $u\in C^{1,\alpha}(\Omega')$ with $\Omega'\Subset \Omega$.
%\end{theorem}

\begin{remark}\label{rmk:planetargetreflector}\rm
We show here that when the target $\Sigma$ is contained in a non-vertical hyperplane, then condition \eqref{eq:AWPreflectors} holds.
Notice that from \eqref{eq:formulaforinnerprodofp}, it is sufficient to show that
\begin{equation}\label{eq:formulafornonverticalplane}
\left\langle D_v^2\left(\dfrac{1+|v|^2}{s_{X}\left( \Lambda(v)\right)} \right)\xi,\xi\right\rangle \leq -C
\end{equation}
for all $v\in \R^n$ and $|\xi|=1$ ($|\eta|=1$).
Let $w\in \R^n$, and $\Sigma=\{(y,y_{n+1}), y_{n+1}=y\cdot w\}$.
Suppose $0<C_1\leq x_{n+1}-x\cdot w\leq C_2$ for all $X\in C_\Omega$.
We calculate $\dfrac{1+|v|^2}{s_X(\Lambda (v))}$. 
From \eqref{eq:reflectedrayLambda} we have
\begin{align*}
y-x&=s_X(\Lambda (v))\,\dfrac{2v}{1+|v|^2}\\
y_{n+1}-x_{n+1}&= s_X(\Lambda (v))\,\dfrac{|v|^2-1}{1+|v|^2}.
\end{align*}
Hence
$\dfrac{1+|v|^2}{s_X(\Lambda (v))}=\dfrac{2 v\cdot w +1 -|v|^2}{x_{n+1}-x\cdot w}$, and then
$\left\langle D_v^2\left(\dfrac{1+|v|^2}{s_{X}\left( \Lambda(v)\right)} \right)\xi,\xi\right\rangle=-\dfrac{2}{x_{n+1}-x\cdot w}\leq -C$.

If $(y,y_{n+1})\in \Sigma$ and $\Sigma$ is a vertical plane, then $y\cdot w=0$ and we get
$\dfrac{1+|v|^2}{s_X(\Lambda (v))}=-\dfrac{2v\cdot w}{x\cdot w}$, so $D_v^2\left(\dfrac{1+|v|^2}{s_{X}\left( \Lambda(v)\right)} \right)=0$.
\end{remark}

\begin{remark}\label{rmk:definitionofreflectorenclosing}\rm
Similarly to Section \ref{sec:alternativedefinitionofrefractor}, we can define parallel reflector by taking paraboloids enclosing the solution.
In other words, we can say $u:\Omega\rightarrow [0,M]$ is a parallel reflector from $\Omega$ to $\Sigma$ if for each $x_0\in\Omega$,
there exists $Y\in\Sigma$ such that $u(x)\leq p(x,Y,X_0)$ for all $x\in\Omega$, where $X_0=(x_0,u(x_0))$.
The same regularity results hold with this definition of solution by changing accordingly the inequalities in the conditions on the target.
Indeed, with the notation in Subsection \ref{subsect:hypothesesontargetrflector}, condition \eqref{eq:reflectorconvexityconditiononthetargetbisbis}
is replaced by 
\begin{equation}\label{eq:essentialhypothesisreflectorbis}
\min\left\{p(x,\bar Y,Z),p(x,\hat Y,Z)\right\}\leq p(x,Y_{Z}(\lambda),Z)-C_{X_0}\,|\bar Y-\hat Y|^{2}|x-z|^{2}.
\end{equation}
The analogue of condition \eqref{eq:AWPreflectors} is now
\begin{equation*}
\dfrac{d^{2}}{d\epsilon^{2}}\left. \left\langle D_x^{2}p(x_0,Y_\epsilon,X_0)\eta,\eta\right\rangle \right|_{\epsilon=0}\leq -C|\xi|^{2}|\eta|^{2},
\end{equation*}
%\begin{equation}\label{eq:inequalityconsequenceofessentialhypothesisreflectorbis}
%%(1-\lambda)p(x,\bar Y,X_0)+\lambda p(x,\hat Y,X_0)=
%p(x,Y^*(\lambda),X_0)\leq p(x,Y(\lambda),X_0)-C\lambda(1-\lambda)|\bar Y-\hat Y|^{2}|x-x_0|^{2},
%\end{equation}
and as before this is equivalent to  \eqref{eq:essentialhypothesisreflectorbis}.
The concavity of $u$ now replaces Lemma \ref{lm:estimateconvexcombinationreflector}, 
the inequality in Lemma \ref{lm:mainlemmareflector} is replaced by 
$$u(x)-p(x,Y,X_0^{\star})\leq  C\,|\bar Y-\hat Y||\bar x-\hat x|+C\,|Y(\lambda)-Y||x-x_0|-C_1\lambda(1-\lambda)|\bar Y-\hat Y|^{2}|x-x_0|^{2},$$
and in the proof the $\max$ is replaced by the $\min$.
In contrast with Remark \ref{rmk:planetargetreflector}, in this case a target contained in a hyperplane cannot satisfy \eqref{eq:essentialhypothesisreflectorbis}.
\end{remark}

\providecommand{\bysame}{\leavevmode\hbox to3em{\hrulefill}\thinspace}
\providecommand{\MR}{\relax\ifhmode\unskip\space\fi MR }
\providecommand{\MRhref}[2]{%
  \href{http://www.ams.org/mathscinet-getitem?mr=#1}{#2}
}
\providecommand{\href}[2]{#2}

%\bibliographystyle{amsalpha}
%\bibliography{/Users/gutierre/Library/texmf/bibtex/bib/monamp}
%\nocite{karakhanyan-wang:nearfieldreflector}
%\bibliography{monamp}
%\nocite{loeper:actapaper}
\end{document}